\documentclass[12pt,leqno]{amsart}
\usepackage[a4paper,left=3cm,right=3cm]{geometry}
\usepackage[latin1]{inputenc}
\usepackage[english]{babel}
\usepackage{amsmath}
\usepackage{verbatim}
\usepackage{amssymb}
\usepackage{enumitem}
\usepackage{color}
\usepackage[table,xcdraw,svgnames,dvipsnames]{xcolor}
\usepackage[colorlinks]{hyperref}
\hypersetup{
	allbordercolors=.,
	citecolor=DarkGreen,
	linkcolor=DarkRed,
	urlcolor=NavyBlue
}
\usepackage{amsthm}
\usepackage{tikz}
\usepackage{tikz-3dplot}
\usepackage{float}
\usepackage{amsfonts}
\usepackage{esint} 
\usepackage[makeroom]{cancel}
 
\newtheorem{theorem}{Theorem}[section]

\newtheorem{lemma}[theorem]{Lemma}
\newtheorem{remark}[theorem]{Remark}
\newtheorem{definition}[theorem]{Definition}

\newcommand{\bo}[1]{{\bf#1}}

\newcommand{\Om}{\Omega}

\newcommand{\RR}{\mathbb R}

\begin{document}
\title[Optimization of Neumann eigenvalues]{Optimization of Neumann Eigenvalues under convexity and geometric constraints}
\author{Beniamin Bogosel, Antoine Henrot, Marco Michetti}
\address[Beniamin Bogosel]{Centre de Math\'ematiques Appliqu\'ees, CNRS, \'Ecole polytechnique, Institut Polytechnique de Paris, 91120 Palaiseau, France}
\email{beniamin.bogosel@polytechnique.edu}

\address[Antoine Henrot]{Universit\'e de Lorraine, CNRS, IECL, F-54000 Nancy, France}
\email{antoine.henrot@univ-lorraine.fr}
\address[Marco Michetti]{
Universit\'e Paris-Saclay, CNRS, Laboratoire de Math\'ematiques d'Orsay, F-91450 Orsay France}
\email{marco.michetti@universite-paris-saclay.fr}

\date{}

\maketitle

\begin{abstract}
    In this paper we study optimization problems for Neumann eigenvalues $\mu_k$ among convex domains with a constraint on the diameter
    or the perimeter. We work mainly in the plane, though some results are stated in higher dimension. We study the existence 
    of an optimal domain in all considered cases. We also consider the case of the unit disk, giving values of the index $k$ for which
    it can be or cannot be extremal. We give some numerical examples for small values of $k$ that lead us to state some conjectures.   
\end{abstract}

{\small
	
	\bigskip
	\noindent
	
	\bigskip
	\noindent
	\bigskip
	
\section{Introduction}
Let $\Omega\subset \mathbb{R}^d$ be a domain (a connected open set).  We consider the classical eigenvalue problem
for the Laplacian with Neumann boundary conditions:
\begin{equation}\label{Neumann eigenvalue problem}
\begin{cases}
-\Delta u = \mu u \quad \text{in } \Omega,\\
\partial_n u=0\quad \text{on }\partial \Omega,
\end{cases}
\end{equation}
where $\partial_n$ denotes the directional derivative with respect to $n$, the outward unit normal vector to $\partial\Omega$.
We recall that some mild regularity (e.g. Lipschitz) is required for the Neumann problem \eqref{Neumann eigenvalue problem} to ensure the compactness
embedding from $H^1(\Omega)$ into $L^2(\Omega)$, leading to the variational problem:
$$
\text{find } u\in H^1(\Omega):\quad \int_\Omega \nabla u\cdot \nabla \varphi = \mu \int_\Omega u\varphi\quad \text{for all }\varphi\in H^1(\Omega).
$$
We will denote by $0=\mu_0(\Omega) < \mu_1(\Omega) \leq \mu_2(\Omega) \leq \ldots$ the sequence of eigenvalues counted with their multiplicity.
In this paper, we are interested in extremum problems for the eigenvalues $\mu_k(\Omega)$ under constraints on the diameter or the perimeter of the set $\Omega$. We will denote by $P(\Omega)$ and $D(\Omega)$ the perimeter and the diameter of the domain $\Omega$.
Let us note that similar problems for Dirichlet eigenvalues have been considered in \cite{BuBuHe09} and \cite{dPV14} for the perimeter constraint
and in \cite{Dirichlet-Diameter} for the diameter constraint. For Steklov eigenvalues, the diameter constraint has been considered in \cite{Steklov-Diameter}. For more general results on optimization problems for eigenvalues, we refer to the books \cite{H06} and \cite{He-book17} and references therein.
In this paper, we will consider only convex domains since, otherwise, the problems are trivial in the sense that
$$\inf\{\mu_k(\Omega) :P(\Omega)=P_0\} = 0 \quad \mbox{and}\quad \sup\{\mu_k(\Omega) :P(\Omega)=P_0\} = +\infty$$
and
$$\inf\{\mu_k(\Omega) :D(\Omega)=D_0\} = 0 \quad \mbox{and}\quad \sup\{\mu_k(\Omega) :D(\Omega)=D_0\} = +\infty.$$
The fact that the infimum is zero in both cases is easily obtained by constructing a sequence of domains approaching a union of $k+1$
disjoint balls. To see that the supremum is infinite with a perimeter constraint, one can think to a fractal-type set  (see also the construction
proposed in \cite{HLL}). With a diameter constraint, an example of a sequence of plane domains for which $D^2(\Omega)\mu_k(\Omega) \to + \infty$ is presented in the paper \cite{HM23}.

Let us remark that, by $-2$-homogeneity of the Neumann eigenvalues, it is equivalent to minimize and maximize $\mu_k(\Om)$ with a constraint on the
perimeter or the diameter or to minimize and maximize the scale invariant quantities $P^{2/(d-1)}(\Omega) \mu_k(\Omega)$ or $D^2(\Omega) \mu_k(\Omega)$.

Let us detail now the different existence results we are able to get for these problems. We work mainly in the two-dimensional case, though 
Theorem \ref{thExSupP3} gives an existence result in three dimensions and Theorem \ref{thNonExInfPd} states a non-existence result in any 
dimension $d\geq 3$. We recall that the minimization problems for $P^2(\Omega) \mu_1(\Omega)$ or $D^2(\Omega) \mu_1(\Omega)$ have no solutions
with infimum given by
$$\inf \{P^2(\Omega)\mu_1(\Omega),\, \Omega \subset \mathbb{R}^2, \text{ bounded, open and convex } \} = 4\pi^2  $$
and
$$\inf \{D^2(\Omega)\mu_1(\Omega),\, \Omega \subset \mathbb{R}^2, \text{ bounded, open and convex } \} = \pi^2 .$$
This is a consequence of the Payne-Weinberger inequality $D^2\mu_1>\pi^2$, see \cite{PW60} and the fact that this inequality is sharp, a minimizing
sequence being a sequence of rectangles $(0,1)\times (0,1/n)$ with $n$ going to $+\infty$.
It is also well-known, see e.g. \cite{HM23}, that the following problems have no maximizers
$$\sup \{D^2(\Omega)\mu_k(\Omega),\, \Omega \subset \mathbb{R}^d, \text{ bounded, open and convex } \} := C_{k,d}. $$
Nevertheless $C_{k,d}$ is an explicit known constant, for example in dimension two: $C_{k,2}=(2j_{0,1}+(k-1)\pi)^2$
where $j_{0,1}$ is the first zero of the Bessel function $J_0$.

\medskip
Let us come to the existence results: in Section \ref{s:existence} we will prove the following theorems:
\begin{itemize}
\item (Theorem \ref{thExInfD}) Let $k\geq 2$ then there exists a solution for the following minimization problem 
\begin{equation*}
\inf \{D^2(\Omega)\mu_k(\Omega),\, \Omega \subset \mathbb{R}^2, \text{ bounded, open and convex } \}   
\end{equation*}
\item (Theorem \ref{thExInfP}) Let $k\geq 2$ then there  exists a solution for the following minimization problem: 
\begin{equation*}
\inf \{P^2(\Omega)\mu_k(\Omega),\, \Omega \subset \mathbb{R}^2, \text{ bounded, open and convex } \}    
\end{equation*}

\item (Theorem \ref{thExSupP}) There exists a solution for the following maximization problem: 
\begin{equation*}
\sup \{P^2(\Omega)\mu_k(\Omega),\, \Omega \subset \mathbb{R}^2, \text{ bounded, open and convex } \}    
\end{equation*}

\item (Theorem \ref{thExSupP3}) There exists a solution for the following maximization problem: 
\begin{equation*}
\sup \{P(\Omega)\mu_1(\Omega),\, \Omega \subset \mathbb{R}^3, \text{ bounded, open and convex } \}    
\end{equation*}

\item (Theorem \ref{thNonExInfPd}) Let $d\geq 3$ then there are no solutions for the following minimization problem: 
\begin{equation*}
\inf \{P^{\frac{2}{d-1}}(\Omega)\mu_k(\Omega),\, \Omega \subset \mathbb{R}^d, \text{ bounded, open and convex } \}    
\end{equation*}
\end{itemize}
After these existence results, in Section \ref{s:ball} we analyze the optimality properties of the disk in two-dimensions.
We prove in particular that the unit disk $B$ is not a minimizer for $D^2 \mu_k$ (among convex domains) when either $\mu_k$ is a simple eigenvalue or
when $k$ is an integer such that the eigenvalue is double with $\mu_{k}=\mu_{k+1}$, see Theorem \ref{t:ballmultipleD}.
We prove the same results of non-minimality of the unit disk for the problem of minimizing $P^2 \mu_k$.
Finally the unit disk is never a maximizer of $P^2 \mu_k$ among convex domains.

At last, Section \ref{s:numerics} is devoted to present the possible optimal shapes we are able to obtain for our three problems.
In the case of the diameter constraint, we use a discretization of the support function and we present results for $k; 2\leq k\leq 9$.
We observe, in particular, that the optimal shapes for $k \in \{4,7\}$ seem to be disks, while it seems that the optimal shape for $k=2$ has constant width
(being not a disk). Moreover, we observe that all the points $x$ on the boundary of optimal shapes saturate either the convexity constraint 
or the diameter constraint.
 In the case of the perimeter constraint, we use a discretization of the gauge function and we present minimizers for $2\leq k\leq 9$ and maximizers for $1 \leq k \leq 3$.
 In this case, our observations are the following: the maximizer of $\mu_1$ under perimeter constraint found numerically is the square. 
 This confirms a conjecture by 
Laugesen-Polterovich-Siudeja, see the recent paper \cite{HLL} where this conjecture is proved
assuming that $\Omega$ has two axis of symmetry. Note that the equilateral triangle gives exactly the same objective value but seems
harder to get with our numerical procedure. The maximizer of $\mu_2$ 
under perimeter constraint seems to be a rectangle with one side equal to twice the other one. 
Moreover, maximizers under perimeter constraint seem to be polygons. It is tempting to use the methodology described in \cite{LN10},
\cite{LNP12}, \cite{LNP16} to try to prove this fact, but the probable multiplicity of the eigenvalues at an optimal shape prevents to use
second order argument which were the basis of these works.

\section{existence of optimal shapes}\label{s:existence}
In this section we prove the existence results presented in the introduction. 
First of all, since we work with convex domains that are uniformly bounded (by the diameter or perimeter constraints), 
for any minimizing or maximizing sequence, only two situations may happen:
\begin{itemize}
    \item either the sequence converges (for the Hausdorff metric) to a convex open set, and in that case since the geometric quantities
    and the Neumann eigenvalues are continuous for the Hausdorff convergence (see e.g. \cite{HePi}) we immediately get existence;
    \item or the sequence (of the closures) converges to a convex set in a lower dimension. For example, plane convex sets may converge to
    a segment. We will say that the sequence is collapsing to a segment in that case.
\end{itemize}
This is this last possibility that we need to exclude in all our existence proofs. For that purpose, we study
the asymptotic behavior of Neumann eigenvalues on a sequence of collapsing domains as we did in \cite{HM22}, \cite{HM23}. 
In particular, we prove a generalization of the asymptotic results obtained in \cite{HM22}. We define the following class of functions:
\begin{equation*}
\mathcal{L}:=\{h\in L^{\infty}(0,1): h \,\text{ non negative, continuous, concave and } \sup h= 1\}.
\end{equation*}
Let $h\in \mathcal{L}$, we decompose $h$ as the sum of two nonnegative, concave functions $h^+,h^-$: $h=h^++h^-$ and we introduce the following set
\begin{equation*}
\Omega_h=\{(x,y)\in \mathbb{R}^2 \;|\,\, 0\leq x\leq 1, \,\, -h^-(x)\leq y\leq h^+(x) \}.
\end{equation*}
In the sequel, the choice of the decomposition of $h$ is not important.
Now we introduce the following Sturm Liouville eigenvalues:
\begin{definition}[Sturm-Liouville eigenvalues]\label{dMUk} 
Let $h\in \mathcal{L}$ we define the following Sturm Liouville eigenvalues:
\begin{equation*}
\begin{cases}
\vspace{0.3cm}
   -\frac{d}{dx}\big(h(x)\frac{du}{dx}(x)\big)=\mu(h) h(x)u(x)  \qquad  x\in \big(0,1\big) \\
      h(0)\frac{du}{dx}(0)=h(1)\frac{du}{dx}(1)=0.
\end{cases}
\end{equation*}
These eigenvalues admit the following variational characterization:
\begin{equation*}
\mu_k(h)=\inf_{E_k} \sup_{0\neq u\in E_k } \frac{\int_0^1(u')^2hdx}{\int_0^1u^2hdx},
\end{equation*}
where the infimum is taken over all $k$-dimensional subspaces of the Sobolev space $H^1([0,1])$ which are  $L^2$-orthogonal to $h$ on $[0,1]$.
\end{definition}
We start by proving the following lemma 
\begin{lemma}\label{lemAsymptoticmu}
Let $h_\epsilon$ be a sequence of functions in $\mathcal{L}$ that converges in $L^2(0,1)$ to a function $h\in \mathcal{L}$, then 
for any decomposition $h_{\epsilon}=h_{\epsilon}^++h_{\epsilon}^-$ as a sum of two nonnegative concave functions, and we set
\begin{equation*}
\Omega_{\epsilon h_{\epsilon}}=\{(x,y)\in \mathbb{R}^2 \;|\,\, 0\leq x\leq 1, \,\, -\epsilon h_{\epsilon}^-(x)\leq y\leq \epsilon h_{\epsilon}^+(x) \}.
\end{equation*}
we have
\begin{equation*}
\liminf_{\epsilon\rightarrow 0} \mu_k(\Omega_{\epsilon h_{\epsilon}})\geq \mu_k(h).   
\end{equation*}
\end{lemma}
\begin{proof} In this proof we will denote by $C$ a constant that can change line by line but that does not depend on $\epsilon$.

Let $u_{k,\epsilon}$ be a Neumann eigenfunction associated to $\mu_k(\Omega_{\epsilon h_{\epsilon}})$, normalized in such a way that $||u_{k,\epsilon}||_{L^2(\Omega_{\epsilon h_{\epsilon}})}=1$, we define the following function in $H^1(\Omega_{h_\epsilon})$
\begin{equation*}
\overline{u}_{k,\epsilon}(x_1,x_2)=\epsilon^{\frac{1}{2}}u_{k,\epsilon}(x_1,\epsilon x_2)\  \forall (x_1,x_2)\in \Omega_{ h_{\epsilon} }. 
\end{equation*}
We want to prove the following bound
\begin{equation}\label{eqBoundH1}
||\overline{u}_{k,\epsilon}||_{H^1(\Omega_{h_{\epsilon}})}\leq C.
\end{equation}
We start by the bound of $||\nabla \overline{u}_{k,\epsilon} ||_{L^2(\Omega_{h_{\epsilon}})}$,
\begin{equation}\label{eq3bnd}
\int_{\Omega_{h_{\epsilon}}}|\nabla \overline{u}_{k,\epsilon}|^2 dx\leq \epsilon \int_{\Omega_{h_{\epsilon}}} \Big( \frac{\partial {u}_{k,\epsilon}}{\partial x_1} \Big )^2+\frac{1}{\epsilon^2}  \Big( \frac{\partial u_{k,\epsilon}}{\partial x_2} \Big )^2 dx \leq \int_{\Omega_{\epsilon{h_{\epsilon}}}}|\nabla u_{k,\epsilon}|^2 dy = \mu_k(\Omega_{\epsilon h_{\epsilon}}) \leq C
\end{equation}
where we did the change of coordinates $y_1=x_1$, $y_2=\epsilon x_2$ and the last inequality comes from the fact that the eigenvalue
$\mu_k$ is uniformly bounded when the diameter is fixed, see \cite{HM23}. 
We recall that $h_{\epsilon}\rightarrow h$ in $L^2(0,1)$, moreover $h_{\epsilon}\in \mathcal{L}$ and $h\in \mathcal{L}$ so by Lemma $3.6$ in \cite{HM22} we have that
$\mu_k(h_{\epsilon})\rightarrow \mu_k(h)$, in particular we conclude that there exists a constant (independent on $\epsilon$) such that $||\nabla \overline{u}_{k,\epsilon} ||_{L^2(\Omega_{h_{\epsilon}})}\leq C$. Using the same change of variable we obtain $||\overline{u}_{k,\epsilon}||_{L^2(\Omega_{h_{\epsilon}})}=1$.

Now, since the functions $h_\epsilon$ and $h$ are concave, the $L^2$ convergence implies in fact the uniform convergence of  $h_\epsilon$ to $h$
on every compact subset of $(0,1)$. Therefore,
the domains $\Omega_{h_\epsilon}$ are a sequence of convex domains containing a ball of fixed radius $\frac{1}{4}$ and in particular they satisfy 
a uniform cone condition, see \cite[Proposition 2.4.4]{HePi}. Let $R$ be the rectangle defined by $R=(0,1)\times (-1,1)$, from \cite{J81} we conclude that there exists a sequence of extensions operators  $E_{\epsilon}$ such that 
\begin{align*}
    &E_{\epsilon}:H^1(\Omega_{h_{\epsilon}})\rightarrow H^1(R)\\
    &E_{\epsilon}(f)=f\,\, \text{in}\,\, \Omega_{h_{\epsilon}} \\
    &||E_\epsilon||\leq C.
\end{align*}
Thanks to the properties of the extension operators $E_\epsilon$ with the estimate \eqref{eqBoundH1} we conclude that there exists $\overline{V}_k\in H^1(R)$ such that, up to a sub-sequence, have 
\begin{equation}\label{eqConvergenceEigenf}
E_\epsilon(\overline u_{k,\epsilon}) \rightharpoonup \overline{V}_k \quad \text{in} \quad H^1(R), \qquad\mbox{and strongly in $L^2$}.
\end{equation}
We want to prove that
\begin{equation}\label{eqVknotdependonx2}
    \frac{\partial \overline{V}_k}{\partial x_2}=0\,\, \text{on}\,\, \Omega_h.
\end{equation}
We start by noticing that:
\begin{equation*}
\int_{\Omega_{h_\epsilon}} \big ( \frac{\partial \overline{u}_{k,\epsilon} }{\partial x_2} \big )^2dx=\epsilon^3 \int_{\Omega_{\epsilon h_\epsilon}} \big ( \frac{\partial u_{k,\epsilon} }{\partial x_2} \big )^2dx\leq C \epsilon^2 \rightarrow 0
\end{equation*}
the last inequality coming from \eqref{eq3bnd}.
In particular we have the following equality 
\begin{align}\label{eqDerivVk}
0&=\liminf_{\epsilon\rightarrow 0} \int_{\Omega_{h_\epsilon}} \big ( \frac{\partial \overline{u}_{k,\epsilon} }{\partial x_2} \big )^2dx=\\ \notag
&=\liminf_{\epsilon\rightarrow 0} 
\int_R \chi_{\Omega_{h_\epsilon}} \big ( \frac{\partial E_\epsilon(\overline u_{k,\epsilon}) }{\partial x_2} \big )^2dx-\int_R \chi_{\Omega_h} \big ( \frac{\partial E_\epsilon(\overline u_{k,\epsilon}) }{\partial x_2} \big )^2dx+\int_R \chi_{\Omega_h} \big ( \frac{\partial E_\epsilon(\overline u_{k,\epsilon}) }{\partial x_2} \big )^2dx,
\end{align}
where $\chi_{\Omega}$ is the indicator function of the set $\Omega$. We know that 
\begin{equation*}
   \liminf_{\epsilon\rightarrow 0}  \int_R \chi_{\Omega_h} \big ( \frac{\partial E_\epsilon(\overline u_{k,\epsilon}) }{\partial x_2} \big )^2dx\geq \int_R \chi_{\Omega_h} \big ( \frac{\partial \overline{V}_k  }{\partial x_2} \big )^2dx,
\end{equation*}
because the functional is convex with respect to the gradient variable, see \cite[Section 8.2]{Evans}. Moreover also the following equality holds 

\begin{equation*}
\liminf_{\epsilon\rightarrow 0} \Big | \int_R \big ( \chi_{\Omega_{h_\epsilon}}- \chi_{\Omega_h} \big ) \big ( \frac{\partial E_\epsilon(\overline u_{k,\epsilon}) }{\partial x_2} \big )^2dx\Big |=0,
\end{equation*}
because  the uniform convergence of $h_\epsilon$ to $h$ on every compact subset of $(0,1)$ implies that
$\chi_{\Omega_{h_\epsilon}}\rightarrow \chi_{\Omega_{h}}$ in $L^1(R)$ and $E_\epsilon(\overline u_{k,\epsilon})$ is bounded in $H^1(R)$. From the above estimates and from \eqref{eqDerivVk} we finally have:
\begin{equation*}
    \int_R \chi_{\Omega_h} \big ( \frac{\partial \overline{V}_k  }{\partial x_2} \big )^2dx=0
\end{equation*}
that implies \eqref{eqVknotdependonx2}. Now from the variational formulation of the Neumann eigenvalues:

\begin{align*}
\liminf_{\epsilon\rightarrow 0} \mu_k(\Omega_{\epsilon h_{\epsilon}})=&\liminf_{\epsilon\rightarrow 0} \max_{\beta\in \mathbb{R}^{k+1}}\frac{\sum_i \beta_i^2\int_{\Omega_{\epsilon h_{\epsilon}}}|\nabla u_{i,\epsilon}|^2dx}{\sum_i \beta_i^2\int_{\Omega_{\epsilon h_{\epsilon}}} u_{\epsilon}^2dx} \\
\geq & \liminf_{\epsilon\rightarrow 0} \max_{\beta\in \mathbb{R}^{k+1}}\frac{\sum_i \beta_i^2\int_{\Omega_{h_{\epsilon}}}|\nabla \overline{u}_{i,\epsilon}|^2dx}{\sum_i \beta_i^2\int_{\Omega_{ h_{\epsilon}}} \overline{u}_{\epsilon}^2dx} \\
\geq & \liminf_{\epsilon\rightarrow 0} \max_{\beta\in \mathbb{R}^{k+1}}\frac{\sum_i \beta_i^2\int_R \chi_{\Omega_{h_\epsilon}} |\nabla E_\epsilon(\overline{u}_{i,\epsilon})|^2dx}{\sum_i \beta_i^2\int_R \chi_{\Omega_{h_\epsilon}} (E_\epsilon{\overline{u}_{\epsilon}})^2dx}.
\end{align*}
We denote by $V_k$ the restriction on the $x_1$ axis of the limit function $\overline{V}_k$, from the convergence in \eqref{eqConvergenceEigenf} and from \eqref{eqVknotdependonx2} we have that for every $i=1,...,k+1$ we have that 
\begin{equation*}
    \lim_{\epsilon\rightarrow 0}\int_R \chi_{\Omega_{h_\epsilon}} (E_\epsilon({\overline{u}_{i,\epsilon}}))^2dx= \int_R \chi_{\Omega_h} \overline{V}_i^2dx=\int_0^1 hV_i^2dx_1.
\end{equation*}
Now using the same argument we used in order to prove \eqref{eqVknotdependonx2} we obtain:
\begin{equation*}
  \liminf_{\epsilon\rightarrow 0} \int_R \chi_{\Omega_{h_\epsilon}} |\nabla E_\epsilon(\overline{u}_{i,\epsilon})|^2dx\geq \int_R \chi_{\Omega_h} |\nabla \overline{V}_i|^2dx= \int_0^1 hV_i'^2dx_1.
\end{equation*}
So we conclude that:
\begin{equation*}
\liminf_{\epsilon\rightarrow 0} \mu_k(\Omega_{\epsilon h_{\epsilon}})\geq   \max_{\beta\in \mathbb{R}^{k+1}}\frac{\sum_i \beta_i^2 \int_0^1 hV_i'^2dx_1}{\sum_i \beta_i^2\int_0^1 hV_i^2dx_1}\geq \mu_k(h),  
\end{equation*}
where the last inequality came from the variational characterization of $\mu_k(h)$.
\end{proof}
We want to find what is the best geometry in which a sequence of domains must collapse in order to get the lowest possible value of the Neumann eigenvalues at the limit. For this reason we are interested in the following minimization problem:
\begin{equation*}
    \inf\{\mu_k(h) \, |\, h\in \mathcal{L} \},
\end{equation*}
in particular we prove the following lemma
\begin{lemma}\label{lemMinSL}
 The following equality holds:
 \begin{equation*}
    \min\{\mu_k(h) \, |\, h\in \mathcal{L} \}=(k\pi)^2
\end{equation*}
the minimizer is given by the function $h\equiv 1$.
\end{lemma}
\begin{proof}
Let $h\in \mathcal{L}$, we denote by $u$ a normalized eigenfunction associated to $\mu_k(h)$. The function $h$ being continuous, $u$ is $C^1$.
The Sturm-Liouville eigenfunctions are Courant sharp, this means that the eigenfunction $u$ has $k+1$ nodal intervals. In particular there are at least $k-1$ points $x_1,...,x_{k-1}$ such that $u'(x_i)=0$ for all $i=1,...,k-1$. We define the following lengths $l_1=x_1,l_2=x_2-x_1,...,l_k=1-x_{k-1}$.

We recall that $u$ solves the equation:
\begin{equation*}
\begin{cases}
\vspace{0.3cm}
   -\frac{d}{dx}\big(h(x)\frac{du}{dx}(x)\big)=\mu(h) h(x)u(x)  \qquad  x\in \big(0,1\big) \\
      h(0)\frac{du}{dx}(0)=h(1)\frac{du}{dx}(1)=0.
\end{cases}
\end{equation*}
We analyze the interval $[0,x_1]$, we know that $u'(x_1)=0$, in particular $u$ is a solution of the following problem:
\begin{equation*}
\begin{cases}
\vspace{0.3cm}
   -\frac{d}{dx}\big(h(x)\frac{du}{dx}(x)\big)=\mu(h) h(x)u(x)  \qquad  x\in \big(0,x_1\big) \\
      h(0)\frac{du}{dx}(0)=h(x_1)\frac{du}{dx}(x_1)=0.
\end{cases}
\end{equation*}
We define the following set:
\begin{equation*}
\Omega_\epsilon=\{(x,y)\in \mathbb{R}^2 \;|\,\, 0\leq x\leq x_1, \,\, 0\leq y\leq \epsilon h(x) \},
\end{equation*}
from Lemma $3.5$ in \cite{HM22} we have that 
\begin{equation}\label{eqLimitNeummx1}
    \mu_k(\Omega_\epsilon)\rightarrow \mu_k(h).
\end{equation}
We claim that 
\begin{equation}\label{eqCond1equalityh}
    \mu_k(h)\geq \frac{\pi^2}{l_1^2}.
\end{equation}
In order to prove this, suppose that 
\begin{equation*}
    \mu_k(h)< \frac{\pi^2}{l_1^2},
\end{equation*}
then from the convergence \eqref{eqLimitNeummx1} we have that there exists an $\overline{\epsilon}$ such that 
\begin{equation*}
 \mu_k(\Omega_{\overline{\epsilon}})<\frac{\pi^2}{D(\Omega_{\overline{\epsilon}})^2},
\end{equation*}
which is a contradiction with the Payne Weinberger inequality \cite{PW60}.

We can apply the same argument to each intervals $[0,x_1], [x_1,x_2],...,[x_{k-1},1]$, obtaining:
\begin{equation*}
   \mu_k(h)\geq \frac{\pi^2}{l_i^2}, \,\,\, \forall \, i=1,...,k.
\end{equation*}
We sum all this relations and we obtain 
\begin{equation*}
   \mu_k(h)\geq \Big ( \frac{1}{k}\sum_{i=1}^{k}\frac{1}{l_i^2} \Big)\pi^2.
\end{equation*}
From the fact that $\sum_{i=1}^{k}l_i=1$ and the convexity of the function $t\mapsto 1/t^2$ it follows
\begin{equation}\label{eqCond2equalityh}
   \frac{1}{k}\sum_{i=1}^{k}\frac{1}{l_i^2}\geq k^2,
\end{equation}
in particular $\mu_k(h)\geq (k\pi)^2$. Moreover we have equality in the inequalities \eqref{eqCond1equalityh} and  \eqref{eqCond2equalityh} only if $h\equiv 1$. 
\end{proof}
We are ready to prove the existence of an open convex domain for the minimization problem under diameter and convexity constraint for eigenvalues with index $k\geq2$. In the case of the first eigenvalue we already know  that the minimizer does not exists and the minimizing sequence is given by a sequence of collapsing rectangles (see \cite{PW60}). 
\begin{theorem}\label{thExInfD}
Let $k\geq 2$ then there exists an open and convex set $\Om^*\subset \RR^2$ such that 
\begin{equation*}
D(\Omega^*)^2\mu_k(\Omega^*)= \inf \{D(\Omega)^2\mu_k(\Omega),\, \Omega \subset \mathbb{R}^2, \text{ bounded, open and convex } \}   
\end{equation*}
\end{theorem}
\begin{proof}
Let $\Omega_\epsilon$ be a minimizing sequence. Thanks to the Blaschke selection theorem we have two possibilities:

\begin{enumerate}
    \item $\Omega_\epsilon$ converge in Hausdorff sense to an open convex set $\Omega$,
    \item the sequence $\Omega_{\epsilon}$ collapse to a segment.
\end{enumerate}
Let us assume that the second outcome happens and denote $\Omega_\epsilon$ the minimizing sequence. Without loss of generality we can assume that $D(\Omega_\epsilon)=1$ for all $\epsilon$. We parametrize the sequence of domains via the functions
$h_{\epsilon}\in \mathcal{L}$ such that $h_{\epsilon}=h_{\epsilon}^++h_{\epsilon}^-$ (the particular choice of $h_{\epsilon}^+$ and $h_{\epsilon}^-$ is not important):
\begin{equation*}
\Omega_{\epsilon  h_{\epsilon}}=\{(x,y)\in \mathbb{R}^2 \;|\,\, 0\leq x\leq 1, \,\, -\epsilon h_{\epsilon}^-(x)\leq y\leq \epsilon h_{\epsilon}^+(x) \}.
\end{equation*}
The sequence of functions $h_\epsilon$ are in  $\mathcal{L}$ so up to a subsequence we know that there exists a function $h\in \mathcal{L}$ such that $h_\epsilon \rightarrow h$ in $L^2(0,1)$ (see for instance Lemma $3.1$ in \cite{HM22}). From Lemma \ref{lemAsymptoticmu} and Lemma \ref{lemMinSL} we have that 
\begin{equation*}
    \liminf \mu_k(\Omega_{\epsilon h_\epsilon})\geq  \mu_k(h)\geq (k\pi)^2.
\end{equation*}
and therefore 
\begin{equation*}
    \liminf D^2\mu_k(\Omega_{\epsilon h_\epsilon})\geq (k\pi)^2.
\end{equation*}
Now to complete the proof, it suffices to find a convex domain with $D^2(\Omega)\mu_k(\Omega) < (k\pi)^2$. For that purpose, let us consider the unit square $Q$, since
it is a tiling domain, the Poly\`a inequality, see \cite{Polya} holds true:
$$\mu_k(Q) < 4\pi k/|Q|=4\pi k.$$
This yields $D^2(Q) \mu_k(Q) < 8\pi k$ and this is sufficient to conclude for $k\geq 3$, while for $k=2$, $D^2(Q) \mu_2(Q) =2\pi^2 < 4\pi^2$.
\end{proof}
\begin{theorem}\label{thExInfP}
For $k\geq 2$, there exists an open and convex set $\Om^*\subset \RR^2$ such that 
\begin{equation*}
P(\Omega^*)^2\mu_k(\Omega^*)= \inf \{P(\Omega)^2\mu_k(\Omega),\, \Omega \subset \mathbb{R}^2, \text{ bounded, open and convex } \}   
\end{equation*}
\end{theorem}
\begin{proof}
We argue like the prof of Theorem \ref{thExInfD}. We consider a sequence of collapsing domains $\Omega_n$ such that $D(\Omega_n)=1$, it is easy to check that $P(\Omega_n)\rightarrow 2$. Therefore, we need to find a convex set $\Omega_k$ such that $P(\Omega_k)^2\mu_k(\Omega_k)\leq 4(k\pi)^2 $.
The unit square still works for $k\geq 16/\pi$ with the same argument of Poly\`a inequality. The square also works for $k=3$
($\mu_3(R)=2\pi^2$) and for $k=2,4,5$ we can consider the unit disk whose eigenvalues are respectively ${j'_{1,1}}^2,{j'_{2,1}}^2, {j'_{0,2}}^2$.
\end{proof}
\begin{theorem}\label{thExSupP}
There exists an open and convex set $\Om^*\subset \RR^2$ such that 
\begin{equation*}
P(\Omega^*)^2\mu_k(\Omega^*)= \sup \{P(\Omega)^2\mu_k(\Omega),\, \Omega \subset \mathbb{R}^2, \text{ bounded, open and convex } \}   
\end{equation*}
\end{theorem}
\begin{proof}
We argue like the prof of Theorem \ref{thExInfP}. We consider a sequence of collapsing domains $\Omega_n$ such that $D(\Omega_n)=1$ from \cite{HM23} we know that
\begin{equation*}
\limsup P(\Omega_n)^2\mu_k(\Omega_n)\leq 4(2j_{0,1}+(k-1)\pi)^2,
\end{equation*}
where $j_{0,1}$ is the first zero of the Bessel function $J_0$. Consider the rectangle $\Omega_k=[0,k]\times [0,1]$, its $k$-th eigenvalue is $\pi^2$, so
we have that $P(\Omega_k)^2\mu_k(\Omega_k)=(2k+2)^2\pi^2>4(2j_{0,1}+(k-1)\pi)^2$ .
\end{proof}
\begin{theorem}\label{thExSupP3}
There exists an open and convex set $\Om^*\subset \RR^3$ such that 
\begin{equation*}
P(\Omega^*)\mu_1(\Omega^*)= \sup \{P(\Omega)\mu_1(\Omega),\, \Omega \subset \mathbb{R}^3, \text{ bounded, open and convex } \}   
\end{equation*}
\end{theorem}
\begin{proof}
Let $\Omega_\epsilon$ be a minimizing sequence. Thanks to the Blaschke selection theorem we have three possibilities:
\begin{enumerate}
\item the sequence $\bar{\Omega}_{\epsilon}$ converges in the Hausdorff sense to a segment
\item the sequence $\bar{\Omega}_{\epsilon}$ converges in the Hausdorff sense to a convex domain of codimension $1$ (a plane convex domain).
\item $\Omega_\epsilon$ converges in the Hausdorff sense to an open convex set $\Omega$,
\end{enumerate}
We need to exclude the first two possibilities. 
To exclude the first possibility, we assume that $D(\Omega_\epsilon)=1$ and $\Omega_\epsilon$ collapses to the diameter. Then, by
inclusion of $\Omega_\epsilon$ into a cylinder of radius $O(\epsilon)$ it is straightforward to check that 
$P(\Omega_\epsilon)\rightarrow 0$ and from \cite{HM23} we have $\mu_1(\Omega_\epsilon)\leq 4\pi^2$ and so $\limsup P(\Omega_\epsilon)\mu_1(\Omega_\epsilon)=0$. Therefore the first eventuality is excluded. 

We want to exclude the second eventuality. Consider the minimizing sequence $\Omega_\epsilon$ such that $D(\Omega_\epsilon)=1$ and $\bar{\Omega}_\epsilon\rightarrow \bar{\Omega}$, without loss of generality we can assume that $\bar{\Omega}$ is included in the plane $\{ z=0\}$. Up to translation, we can parameterize the sequence of collapsing domains in the following way:
\begin{equation*}
\Omega_\epsilon=\{(x,y,z)\in \mathbb{R}^3 \;|\,\, (x,y)\in \Omega_{\epsilon}\cap \{z=0\}, \,\, -\epsilon h_{\epsilon}^-(x,y)\leq z\leq \epsilon h_{\epsilon}^+(x,y) \},
\end{equation*}
where $h_{\epsilon}^+$ and $h_{\epsilon}^-$ are two positive concave functions with supports equal to $\Omega_{\epsilon}\cap \{z=0\}$ and we define $h_\epsilon=h_{\epsilon}^++h_{\epsilon}^-$ to be a concave function with $\|h_\epsilon\|_\infty=1$. 

Let us denote by $S_f$ the support of a function $f$. By using test functions that depends only on the first two coordinates in the variational characterization we obtain:
\begin{equation*}
P(\Omega_{\epsilon})\mu_1(\Omega_\epsilon)\leq P(\Omega_{\epsilon})\mu_1(h_\epsilon),
\end{equation*}
where, for a function $h$ depending on the two variables $x,y$:
\begin{equation}\label{d:mu1h}
\mu_1(h):= \inf \Big \{\frac{\int_{S_h}h|\nabla u|^2}{\int_{S_h}h u^2}\,\,\Big | u\in H^1(S_h), \int_{S_h} h u=0\Big \}.
\end{equation}
As the Neumann eigenvalues are translational invariant, for every $\epsilon$ we can choose an origin in the plane $(x,y,0)$ such that 
$\int_{S_{h_\epsilon}} h_\epsilon x=\int_{S_{h_\epsilon}} h_\epsilon y=0$. Using the coordinate functions in the definition \eqref{d:mu1h} we get:
\begin{equation}\label{e:Pmu3}
P(\Omega_{\epsilon})\mu_1(\Omega_\epsilon)\leq 2P(\Omega_{\epsilon})\frac{\int_{S_{h_\epsilon}} h_\epsilon}{\int_{S_{h_\epsilon}}h_\epsilon(x^2+y^2)}.
\end{equation}
 
From the assumption that we have on the diameter we have that $h_\epsilon=0$ on $\partial \Omega_\epsilon \cap \{z=0\}$. We consider now an extension of the function we just defined: let  $\bar{h}_\epsilon=h_\epsilon$ be the function $h_\epsilon$ extended by zero outside the support (recall that $h_\epsilon=0$ on the boundary of its support)  now we can extract a sub sequence such that $\bar{h}_\epsilon\rightarrow \bar{h}$ (in the $L^\infty weak-*$ sense), 
where $\bar{h}$ is the extension of a function $h$ such that $h$ is non negative on its support and is zero on the boundary of its support. It is straightforward to check that $\limsup P(\Omega_\epsilon)\leq 2 \mathcal{H}^2(supp(h))$, where $\mathcal{H}^2$ is the two dimensional Hausdorff measure. From \eqref{e:Pmu3} and from the fact that $h_\epsilon\rightarrow h$ in $L^2(\mathbb{R}^2)$ we obtain 
\begin{equation}\label{e:limsupPM}
\limsup P(\Omega_{\epsilon})\mu_1(\Omega_\epsilon) \leq 4 \mathcal{H}^2(S_h) \frac{\int_{S_h} h}{\int_{S_h}h(x^2+y^2)}.
\end{equation}
In order to obtain an upper bound in \eqref{e:limsupPM} we want to solve the following problem:
\begin{equation}\label{e:momentInertia}
\min \Big \{\frac{\int_{S_h}h(x^2+y^2)}{\int_{S_h}h} \Big | ||h||_{L^\infty(\mathbb{R}^2)}=1, \,\, h \text{ non negative, concave } S_h\subset B_1(0)  \Big \}.
\end{equation}
Where $B_1(0)$ the ball of radius $1$ in the plane $\mathbb{R}^2$ centered at the origin. We stress the fact that $||h||_{L^\infty(\mathbb{R}^2)}=1$ is not a constraint because the functional is invariant under multiplication of $h$ by a constant and the fact that $supp(h)\subset B_1(0)$ came from the choice 
of the minimizing sequence that satisfies $D(\Omega_\epsilon)=1$.

We want to pass from \eqref{e:momentInertia} to a one dimensional problem. We denote by $f^*$ and $f_*$ respectively the spherical decreasing rearrangement and the  spherical increasing rearrangement of the function $f$. We use the classical inequality, see e.g. \cite{Kawohl} $\int fg\geq \int f^*g_*$, noticing that $(x^2+y^2)_*=(x^2+y^2)$ we finally obtain:
\begin{equation*}
\int_{S_h}h(x^2+y^2)\geq \int_{S_{h^*}}h^*(x^2+y^2),
\end{equation*}
where $h^*$  is a radial non negative function such that $h(0)=1$ and $S_{h^*}$ is a ball centered in $(0,0)$ with the same measure as $S_h$. Moreover $h$ is a positive concave function, which implies that  $\{(x,y,z)\in \mathbb{R}^3|(x,y)\in S_h\, , 0\leq z\leq h(x,y)\}$ is convex. After the rearrangement we also have that   $\{(x,y,z)\in \mathbb{R}^3|(x,y)\in S_{h^*}\, , 0\leq z\leq h^*(x,y)\}$ is a convex set so $h^*$ is also a concave function. For more details regarding symmetrization of convex bodies see \cite[p. 77-78]{bonnesen-fenchel}. 
Therefore, after this rearrangement argument, we are led to solve the following one dimensional optimization problem:
\begin{equation}\label{e:momentInertia1}
\min \Big \{ \frac{\int_0^R h(r)r^3dr}{\int_0^R h(r)rdr} \Big | h(0)=1,\, h \text{ concave} , h \text{ decreasing}, h\geq 0, \,  \Big \}.
\end{equation}
We want to show now that the solution of the problem \eqref{e:momentInertia1} is given by a linear function $\bar{h}$.
First of all, existence of a minimizer $\hat{h}$ is straightforward. Note that, by homogeneity of the functional, we can replace the constraint
$h(0)=1$ by an integral constraint like $\int_0^R h(r)r dr=1$.

 Now we want to write the optimality condition.
For that purpose, we use the abstract formalism developed in \cite{LN10}. The concavity constraint is expressed by a Lagrange multiplier
that is here a function $\xi \in H^1(0,1)$ such that $\xi\geq 0$ and $\xi=0$ on the support of the measure $({\hat{h}}'')$. Moreover 
the constraint $h$ decreasing is equivalent in that context to $h'(0)\leq 0$ and $h\geq 0$ is equivalent to $h(R)\geq 0$.
Therefore, there are also two measures $\mu_0$ with support at $r=0$ and $\mu_R$ with support at $r=R$ such that the optimality
condition writes
\begin{equation}\label{e:optimalityxi}
-\xi''=r^3 + a r +\mu_0 + \mu_1
\end{equation}
the term $a r$ coming from the linear constraint $\int_0^R h(r)r dr=1$.
Let us restrict to the open interval $(0,R)$ (where the measures $\mu_0$ and $\mu_R$ disappear), we have the ODE
\begin{equation}\label{ODEoptim}
-\xi''=r^3 + a r \,,\,\, r\in (0,1) .
\end{equation}
Let us denote by $S$ the support of the measure ${\hat{h}}''$. Our aim is to prove that $ S\cap (0,1)$ is empty which will show that ${\hat{h}}$ is linear on $(0,1)$.

The first step is to prove that $S$ does not contains any interval. Suppose that $(\alpha, \beta)\subset S$ from \eqref{ODEoptim} and the definition of
$\xi$, we obtain that $r^3+a r=0$ for all $r\in (\alpha, \beta)$ and this is a contradiction. 

Now let $\alpha$, $\beta \in [0,1]$ such that $(\alpha,\beta)$ is a maximal interval in the open set $S^c$. According to \eqref{ODEoptim} we have that :
\begin{equation}\label{e:odexi}
\begin{cases}
-\xi''(r)=r^3+ar \,\, \text{ in }(\alpha,\beta),\\
\xi(\alpha)=\xi(\beta)=0.
\end{cases}
\end{equation}
Now from the ODE, we see that $\xi$ is in $\mathcal{C}^1$ and since $\xi$ must remain nonnegative, we conclude that also $\xi'(\alpha)=\xi'(\beta)=0$.   
From \eqref{e:odexi} we conclude that $\int_\alpha^\beta r^3+ar=\int_\alpha^\beta r(r^3+ar)=0$, and so $r^3+ar$ should have at least two zeros 
inside the interval $(\alpha,\beta)$ this is impossible. We conclude that $ S^c\cap (0,1)$ does not contain an interior interval. Therefore, the only
possibilities is that $ S\cap (0,1)$ has zero or one point. 

Let us exclude this last case. If $ S\cap (0,1)=x_0$ it means that ${\hat{h}}$ is piecewise affine with a change of slope at $x_0$. If we denote by
$b={\hat{h}}(x_0)$ we see that both $\int_0^R {\hat{h}}(r)r^3 dr$ and $\int_0^R {\hat{h}}(r)rdr$ are affine functions of $b$. 
Therefore, the functional we want
to minimize is homographic in $b$. Therefore it is 
\begin{itemize}
    \item either strictly monotone and we can improve the value of the functional by moving vertically the point $(x_0,b)$ showing that it is not an optimum
    \item or the function is constant in $b$ and we can move $b$ down to the position where ${\hat{h}}$ becomes linear without changing the value of the functional
\end{itemize}
Therefore, in any case we have proved that the minimizer is a linear function, say ${\hat{h}}(r)=c-\frac{c-d}{R} r$. Now a straightforward calculation gives
\begin{equation}
  \int_0^R h(r)r^3 dr = \frac{R^4}{5} \left(d+\frac{c}{4}\right) \quad \int_0^R h(r)r dr = \frac{R^2}{3} \left(d+\frac{c}{2}\right) .
\end{equation}
Now the ratio is clearly minimized when taking $d=0$ what yields
\begin{equation*}
\frac{\int_0^R {\hat{h}}(r) r^3dr}{\int_0^R {\hat{h}}(r) rdr}=\frac{3R^2}{10},
\end{equation*}
combining this with \eqref{e:limsupPM} and the area being given by $\pi R^2$, we finally obtain:
\begin{equation}
\limsup P(\Omega_{\epsilon})\mu_1(\Omega_\epsilon) \leq \frac{40}{3}\pi.
\end{equation}
Consider now the unit cube $[0,1]^3$ we have that $P([0,1]^3)\mu_1([0,1]^3)=6\pi^2 >\frac{40}{3}\pi$, this concludes the proof.
\end{proof}
\begin{remark}
    We do not know whether the cube is the maximizer of $P(\Omega)\mu_1(\Omega)$ in dimension 3. By analogy with the two-dimensional case
    (where the square and the equilateral triangle are conjectured to be the maximizers of $P^2\mu_1$, see \cite{HLL}), this is a reasonable conjecture.
\end{remark}

\begin{theorem}\label{thNonExInfPd}
Let $d\geq 3$ then there are no solutions of the following minimization problem: 
\begin{equation*}
\inf \{P^{\frac{2}{d-1}}(\Omega)\mu_k(\Omega),\, \Omega \subset \mathbb{R}^d, \text{ bounded, open and convex } \}    
\end{equation*}
\end{theorem}
\begin{proof}
We can easily exhibit a minimizing sequence of convex domains $\Omega_n$ such that $P^{\frac{2}{d-1}}(\Omega_n)\mu_k(\Omega_n)$ goes to zero.
Indeed, take for example a cuboid $(0,1/n)^{d-1} \times (0,1)$: its perimeter goes to zero with $n$ while its $k$-th eigenvalue
$\mu_k(\Omega_n)$ is uniformly bounded, since its diameter is bounded
  (see for instance \cite{HM23}). 
\end{proof}

\section{analysis of the disk}\label{s:ball}
In this section we consider the case of the disk and we study the optimality conditions around the disk. We start by studying the optimization problem under diameter constraint.
\begin{theorem}\label{t:ballmultipleD}
Let $B\subset \mathbb{R}^2$ be the unit disk and let $k\in \mathbb{N}$ be an index such that $\mu_{k-1}(B)=\mu_k(B)$ then $B$ is never a local minimizer for the problem:
\begin{equation*}
\inf \{D(\Omega)^2\mu_{k-1}(\Omega),\, \Omega \subset \mathbb{R}^2, \text{ bounded, open and convex } \}   
\end{equation*}
\end{theorem}
\begin{proof}
To simplify the notation we introduce $\mu_{k-1}(B)=\mu_k(B)=\omega_0^2={j'}_{m,l}^2$ that is the square of a zero of the derivative of the Bessel function $J_m$. 
We construct $\Omega_\epsilon$ a convex perturbation of the unit disk by perturbing the support function of the unit disk. The support function of $\Omega_\epsilon$ is given by $h_\epsilon(\theta)=1+\epsilon f(\theta)$, consequently the distance from the origin to a point in $\partial \Omega_\epsilon$ is given by $r_\epsilon=1+\epsilon f(\theta)+o(\epsilon)$. We want to find a first order expansion of the eigenvalue $\mu_{k-1}(\Omega_{\epsilon})$ and prove that for a particular choice of $f$ we have  
$\mu_{k-1}(\Omega_{\epsilon})<\mu_{k-1}(B)$.  
We introduce $u_1$ and $u_2$ the eigenfunctions corresponding to $\mu_{k-1}(B)$ and $\mu_k(B)$:
\begin{align*}
u_1(r,\theta)&=A J_m(j_{m,l}'r)\cos(m\theta)\\
u_2(r,\theta)&=A J_m(j_{m,l}'r)\sin(m\theta),
\end{align*}
where $J_m$ is the Bessel function of index $m$, $j_{m,l}'$ is the $l$th-zero of the function  $J_m'$ and $A$ is a normalizing constant such that $||u_i||_{L^2(B)}=1$ with $i=1,2$. For multiple eigenvalues, we know that, see e.g. \cite[Chapter 2]{H06}, $\mu_{k-1}(\Omega_{\epsilon})=\mu_{k-1}(B)+\epsilon\lambda_1+o(\epsilon)$ where $\lambda_1$ is the smallest eigenvalue of the following matrix
\begin{equation*}
\mathcal{M}=\begin{pmatrix}\vspace{0.5cm}
\int_{\partial B}(|\nabla u_1|^2-\omega_0^2u_1^2)fds & \int_{\partial B}(\nabla u_1\cdot\nabla u_2 -\omega_0^2u_1u_2)fds \\ 
\int_{\partial B}(\nabla u_1\cdot\nabla u_2 -\omega_0^2u_1u_2)fds & \int_{\partial B}(|\nabla u_2|^2-\omega_0^2u_2^2)fds
\end{pmatrix}
\end{equation*}
We expand the function $f$ in Fourier series $f(\theta)=\sum_{p=0}^{\infty}\alpha_p\cos(p\theta)+\beta_p\sin(p\theta)$ and using the explicit expression of $u_1$ and $u_2$ we obtain:
\begin{equation*}
\mathcal{M}=A^2 \pi J_m(\omega_0)\begin{pmatrix}\vspace{0.5cm}
(m^2-\omega_0^2)\alpha_0-\frac{m^2+\omega_0^2}{2}\alpha_{2m} & -\frac{m^2+\omega_0^2}{2}\beta_{2m} \\ 
-\frac{m^2+\omega_0^2}{2}\beta_{2m} & (m^2-\omega_0^2)\alpha_0+\frac{m^2+\omega_0^2}{2}\alpha_{2m}
\end{pmatrix}
\end{equation*}
the smallest eigenvalue of $\mathcal{M}$ is given by:
\begin{equation}\label{e:firstorderN}
\lambda_1=-\omega_0^2 \left (2\alpha_0+\frac{m^2+\omega_0^2}{|m^2-\omega_0^2|}\sqrt{\alpha_{2m}^2+\beta_{2m}^2} \right).
\end{equation}
On the other hand, the diameter satisfies $D(\Omega_\epsilon)=2+\epsilon M_f$ where $M_f=\sup_{\theta\in [0,\pi]}(f(\theta)+f(\theta+\pi))$, using \eqref{e:firstorderN} we finally obtain:
\begin{equation}\label{e:firstorderDN}
D(\Omega_{\epsilon})^2\mu_{k-1}(\Omega_\epsilon)=4\omega_0^2+\epsilon 4 (M_f\omega_0^2+\lambda_1)+o(\epsilon).
\end{equation}
From \eqref{e:firstorderN} and  \eqref{e:firstorderDN} we can conclude if we can find a function $f(\theta)$ such that:
\begin{equation*}
M_f-2\alpha_0-\frac{m^2+\omega_0^2}{|m^2-\omega_0^2|}\sqrt{\alpha_{2m}^2+\beta_{2m}^2}<0.
\end{equation*}
We consider $f(\theta)=\alpha_0+\cos(2m\theta)+\phi(\theta)$, where $\phi$ is a function satisfying
\begin{itemize}
    \item $\int_0^{2\pi}\phi(\theta)\cos(2m\theta)d\theta =0$,
    \item $M_f<2\alpha_0+1+\eta$, where $\eta$ is chosen later.
\end{itemize} 
It is straightforward to check that one can choose a function $\phi$ that is $\pi/m$ periodic
and piece-wise affine that will be convenient.
With this precise choice of $f$ we obtain:
\begin{equation*}
M_f-2\alpha_0-\frac{m^2+\omega_0^2}{|m^2-\omega_0^2|}\sqrt{\alpha_{2m}^2+\beta_{2m}^2}<1+\eta-\frac{m^2+\omega_0^2}{|m^2-\omega_0^2|}<0,
\end{equation*}
where we choose $\eta<\frac{m^2+\omega_0^2}{|m^2-\omega_0^2|}-1$. 
\end{proof}

\begin{theorem}\label{t:ballsingleD}
Let $B\subset \mathbb{R}^2$ be the unit disk and let $k\in \mathbb{N}$ be an index such that $\mu_k(B)$ is a simple eigenvalue, 
then $B$ is never a local minimizer for the problem:
\begin{equation*}
\inf \{D(\Omega)^2\mu_k(\Omega),\, \Omega \subset \mathbb{R}^2, \text{ bounded, open and convex } \}   
\end{equation*}
\end{theorem}
\begin{proof}
In the previous theorem, first order optimality condition was enough to conclude to the non-minimality of the disk.
For a simple eigenvalue, it turns out that the first order derivative is non-negative and we need to work with deformation for which
this first derivative is zero and, then look at the second order derivative in order to conclude.
Therefore, we proceed in a different way with respect the proof of Theorem \ref{t:ballmultipleD}. Indeed we will not use a shape derivative approach, but we will expand the normal derivative of the perturbed eigenfunction on the boundary. As in the proof of Theorem \ref{t:ballmultipleD} we perturb the disk $B$ by perturbing the support function. Let $\Omega_\epsilon$ be a domain with support function given by $h_\epsilon(t)=1+\epsilon f(t)$. 
From the formulae giving the parameterization of the boundary: $x(t)=h_\epsilon(t) \cos t - h'_\epsilon(t) \sin t; y(t)=h_\epsilon(t) \sin t + 
h'_\epsilon(t) \cos t$ we infer 
that the distance from the origin to a point in $\partial \Omega_\epsilon$ is given by
\begin{equation}
r_\epsilon=1+\epsilon f(t)+\frac{\epsilon^2}{2}f'(t)^2+o(\epsilon^2).
\end{equation}
Let $\theta_\epsilon$ be the polar angle (we recall that $t$ be the normal angle) then we have
\begin{equation}
\theta_\epsilon=t+\epsilon f'(t)-\epsilon^2f(t)f'(t). 
\end{equation}
We introduce $\mu_k(\Omega_\epsilon)=\omega_\epsilon^2$, thanks to the fact the eigenvalue is simple we have that $\omega_\epsilon$ admits the following expantion $\omega_\epsilon=\omega_0+\epsilon \omega_1+\epsilon^2\omega_2+o(\epsilon^2)$ where $\omega_0^2=j_{0,l}'^2$ with $l\geq 2$. The aim is to compute $\omega_1$ and $\omega_2$. We write the eigenfunction $u_\epsilon$ of $\mu_k(\Omega_\epsilon)$ as an expansion in the basis given by the eigenfunctions of the disk
\begin{equation}
u_\epsilon(r,\theta)=\sum_{n=0}^\infty A_n(\epsilon)J_n(\omega_\epsilon r_\epsilon)\cos(n\theta_\epsilon)+B_n(\epsilon)J_n(\omega_\epsilon r_\epsilon)\sin(n\theta_\epsilon)
\end{equation} 
where 
\begin{align*}
A_n(\epsilon)&=\delta_{n,0}a_n+\epsilon b_n+\epsilon^2c_n\\
B_n(\epsilon)&=\delta_{n,0}\bar{a}_n+\epsilon \bar{b}_n+\epsilon^2\bar{c}_n,
\end{align*}
where $a_n,b_n,c_n,\bar{a}_n,\bar{b}_n$ and $\bar{c}_n$ are real numbers and $\delta_{n,0}$ is a Kronecker delta. We want to impose the relation $\partial_n u_\epsilon|_{\partial \Omega_\epsilon}=0$ and identifying the main term, the term in $\epsilon$ and the term in $\epsilon^2$ finding in this way explicit formulas for $\omega_1$ and $\omega_2$. In order to do this computation we first need the following expansion of the Bessel function $J_0'$ around $\omega_\epsilon r_\epsilon$ we obtain 
\begin{align}\notag
J'_0(\omega_\epsilon r_\epsilon)&=\cancel{J'_0(\omega_0)}+\epsilon J''_0(\omega_0)(\omega_1+\omega_0f(t))+\epsilon^2 G_0(t)+o(\epsilon^2),\\ \label{e:G0}
G_0(t)&:=J''_0(\omega_0)(\omega_2+\omega_1f(t)+\frac{\omega_0}{2}f'(t)^2)+\frac{J'''_0(\omega_0)}{2}(\omega_1+\omega_0f(t))^2.
\end{align} 
As explained before, we realize that the first order expansion is zero if and only if $\alpha_{2k}=\beta_{2k}=0$, we decide to make the following choice of the perturbation:
\begin{equation}\label{e:fexpand}
f(t)=\alpha_0+\sum_{k=1}^\infty\alpha_{2k+1}\cos((2k+1)t)+\beta_{2k+1}\sin((2k+1)t).
\end{equation} 

We now compute an expansion up to $\epsilon^2$ of $\partial_n u_\epsilon|_{\partial \Omega_\epsilon}=0$ and we obtain:
\begin{align} \label{e:bigeq}
0=\sum_{n=0}^\infty&\left[\omega_0+\epsilon \omega_1+\epsilon^2\omega_2\right]\left[J'_n(\omega_0)+\epsilon J''_n(\omega_0)(\omega_1+\omega_0f(t))+\epsilon^2 G_n(t)\right]\\ \notag
&\times\Big[(\delta_{n,0}a_n+\epsilon b_n+\epsilon^2c_n)\left(\cos(nt)-\epsilon nf'(t)\sin(nt)\right)\\ \notag
&+(\delta_{n,0}\bar{a}_n+\epsilon \bar{b}_n+\epsilon^2\bar{c}_n)\left(\sin(nt)+\epsilon nf'(t)\cos(nt)\right) \Big]\\ \notag
&-n\Big[\epsilon f'(t)-2\epsilon^2f(t)f'(t)\Big]\Big[J_n(\omega_0)+\epsilon J'_n(\omega_0)(\omega_1+\omega_0f(t))\Big]\\ \notag
&\times \Big[ -(\delta_{n,0}a_n+\epsilon b_n)\left(\sin(nt)+\epsilon nf'(t)\cos(nt)\right)\\ \notag
&+(\delta_{n,0}\bar{a}_n+\epsilon \bar{b}_n)\left(\cos(nt)-\epsilon nf'(t)\sin(nt)\right) \Big]+o(\epsilon^2).
\end{align}

\textbf{Term in $\epsilon$.} From equation \eqref{e:bigeq}, identifing the term in front of $\epsilon$  we obtain:
\begin{equation}\label{e:termepsilon}
a_0\omega_0J''_0(\omega_0)(\omega_1+\omega_0 f(t))+\omega_0 \sum_{n=1}^\infty J_n'(\omega_0)(b_n\cos(nt)+\bar{b}_n\sin(nt)) =0 \;\;\; \forall t\in [0,2\pi].
\end{equation}
In particular the mean of the above function is zero, using the expansion \eqref{e:fexpand} and identifying the zero term in the expansion we finally obtain:
\begin{equation}\label{e:w1}
\omega_1=-\alpha_0\omega_0.
\end{equation}
Imposing that the coefficients of $\cos((2k+1)t)$ and $\sin((2k+1)t)$ in the Fourier expansion in \eqref{e:termepsilon} are zero and using the fact that $J''_0(\omega_0)=-J_0(\omega_0)$  we obtain:
\begin{align}\label{e:b2k+1}
b_{2k+1}=a_0\omega_0\frac{J_0(\omega_0)}{J'_{2k+1}(\omega_0)}\alpha_{2k+1}\\ \label{e:barb2k+1}
\bar{b}_{2k+1}=a_0\omega_0\frac{J_0(\omega_0)}{J'_{2k+1}(\omega_0)}\beta_{2k+1}
\end{align} 
\textbf{Term in $\epsilon^2$} From equation \eqref{e:bigeq}, identifying the term in front of $\epsilon^2$  we obtain:
\begin{align}\label{e:termepsilon2}
&\underbrace{a_0\omega_0G_0(t)+a_0\omega_1J''_0(\omega_0)(\omega_1+\omega_0f(t))}_{:=I_1}+\\ \notag
&+\underbrace{f'(t)\sum_{n=1}^\infty n \big[\omega_0J_n'(\omega_0)-J_n(\omega_0) \big]\big[\bar{b}_n\cos(nt)-b_n\sin(nt)\big]}_{:=I_2}+\\ \notag
&+\underbrace{f(t)\sum_{n=1}^\infty n \omega_0^2J_n''(\omega_0)\big[b_n\sin(nt)+\bar{b}_n\cos(nt)\big]}_{:=I_3}+I_4=0 \;\;\; \forall t\in [0,2\pi],\\ \notag
\end{align}
where in $I_4$ we collect all the terms in which the dependence in $t$ is given only by $\cos(nt)$ and $\sin(nt)$, in particular $\int_0^{2\pi}I_4=0$. 

We compute $\int_0^{2\pi}I_1$, Using the expansion \eqref{e:fexpand}, equation \eqref{e:w1}, equation \eqref{e:G0}, Parseval identity in order to compute $(2\pi)^{-1}\int_0^{2\pi}f^2$ and $(2\pi)^{-1}\int_0^{2\pi}f'^2$ and using the relations $J''_0(\omega_0)=-J_0(\omega_0)$  and $J'''_0(\omega_0)=\omega_0^{-1}J_0(\omega_0)$ we finally obtain 
\begin{align*}
\int_0^{2\pi}I_1=&a_0\omega_0\big [-J_0(\omega_0)\big (\omega_2-\alpha_0^2\omega_0+\frac{\omega_0}{4}\sum_{k=1}^\infty(2k+1)^2(\alpha_{2k+1}^2+\beta_{2k+1}^2)\big )+\\
&+\frac{J_0(\omega_0)\omega_0}{4}\big( \sum_{k=1}^\infty(\alpha_{2k+1}^2+\beta_{2k+1}^2) \big) \big ].
\end{align*}
We compute $\int_0^{2\pi}I_2$, using the expansion \eqref{e:fexpand} and the relations \eqref{e:b2k+1} and \eqref{e:barb2k+1} we obtain
\begin{align}
\int_0^{2\pi}I_2&=\sum_{k=1}^\infty \frac{(2k+1)^2}{2}\big [\omega_0J'_{2k+1}(\omega_0)-J_{2k+1}(\omega_0) \big] \big [b_{2k+1}\alpha_{2k+1}+\bar{b}_{2k+1}\beta_{2k+1} \big ] \\ \notag
&=\sum_{k=1}^\infty a_0\omega_0J_0(\omega_0) \frac{(2k+1)^2}{2} \big [\omega_0-\frac{J_{2k+1}(\omega_0)}{J_{2k+1}'(\omega_0)} \big] \big [\alpha_{2k+1}^2+\beta_{2k+1}^2 \big ].
\end{align}
Similarly we obtain
\begin{equation}
\int_0^{2\pi}I_3= \sum_{k=1}^\infty a_0\omega_0^3J_0(\omega_0)\frac{J''_{2k+1}(\omega_0)}{2J'_{2k+1}(\omega_0)}[\alpha_{2k+1}^2+\beta_{2k+1}^2 \big ].
\end{equation}
From \eqref{e:termepsilon2} we have that $\int_0^{2\pi}I_1+I_2+I_3=0$, from the above equations and the relation $\omega_0^2J_{2k+1}''(\omega_0)+\omega_0J_{2k+1}'(\omega_0)+(\omega_0^2-(2k+1)^2)J_{2k+1}(\omega_0)=0$ we finally obtain 
\begin{equation}\label{e:w2}
\omega_2=\alpha_0^2\omega_0+\sum_{k=1}^\infty c_k (\alpha_{2k+1}^2+\beta_{2k+1}^2),
\end{equation} 
where 
\begin{equation}\label{e:ck}
c_k:=\omega_0(k^2+k)-\frac{\omega_0^2J_{2k+1}(\omega_0)}{2J_{2k+1}'(\omega_0)}.
\end{equation}
From the perturbation we have chosen, we have $D(\Omega_{\epsilon})=2+2\epsilon\alpha_0$, using \eqref{e:w1} we obtain:
\begin{equation}
D(\Omega_{\epsilon})^2\mu_k(\Omega_\epsilon)=
4\omega_0^2+\epsilon^28(\omega_0\omega_2-\alpha_0^2\omega_0^2)+o(\epsilon^2).  
\end{equation}
From \eqref{e:w2} we conclude the proof if we are able to find a particular perturbation $f$ such that
\begin{equation}
\omega_0\omega_2-\alpha_0^2\omega_0^2=\omega_0\sum_{k=1}^\infty c_k (\alpha_{2k+1}^2+\beta_{2k+1}^2)<0.
\end{equation} 
We choose a perturbation such that $\alpha_3=\beta_3=1$, and all the others Fourier coefficient equal to zero. From \eqref{e:ck} we need to prove that 
\begin{equation}\label{e:ratiobessel3}
\frac{J_{3}'(\omega_0)}{J_{3}(\omega_0)}<\frac{\omega_0}{4}.
\end{equation} 
Using the relations $\omega_0J_{3}'(\omega_0)=3J_{3}(\omega_0)-\omega_0J_{2}(\omega_0)$ and $J_{3}(\omega_0)=4J_{2}(\omega_0)\omega_0^{-1}$, we conclude that \eqref{e:ratiobessel3} is true if and only if $\omega_0^2=j_{0,l}'^2>6$, this last inequality is true for all $l\geq 2$.
\end{proof}
\begin{theorem}\label{t:ballmultiplePinf}
Let $B\subset \mathbb{R}^2$ be the unit disk and let $k\in \mathbb{N}$ be an index such that $\mu_{k-1}(B)=\mu_k(B)$ then $B$ is never a local minimizer for the problem:
\begin{equation*}
\inf \{P(\Omega)^2\mu_{k-1}(\Omega),\, \Omega \subset \mathbb{R}^2, \text{ bounded, open and convex } \}   
\end{equation*}
\end{theorem}
\begin{proof}
We mimic the argument of the proof of Theorem \ref{t:ballmultipleD}. To simplify the notation we introduce $\mu_{k-1}(B)=\mu_k(B)=\omega_0^2$. We construct $\Omega_\epsilon$ a convex perturbation of the unit disk by perturbing the support function of the unit disk. The support function of $\Omega_\epsilon$ is given by $h_\epsilon(\theta)=1+\epsilon f(\theta)$.
We consider the Fourier expansion of the perturbation $f(\theta)=\sum_{p=0}^{\infty}\alpha_p\cos(p\theta)+\beta_p\sin(p\theta)$.
Since the perimeter of $\Omega_\epsilon$ is given by $P(\Omega_\epsilon)=\int_0^{2\pi} h_\epsilon(\theta) d\theta$,
the following asymptotic expansion for the perimeter holds $P(\Omega_\epsilon)=2\pi(1+\epsilon \alpha_0)$. The expansion of the Neumann eigenvalue is given by $\mu_{k-1}(\Omega_{\epsilon})=\mu_{k-1}(B)+\epsilon\lambda_1+o(\epsilon)$ where $\lambda_1$ is given by \eqref{e:firstorderN}, we finally obtain:
\begin{equation*}
P(\Omega_\epsilon)^2\mu_{k-1}(\Omega_\epsilon)=4\pi^2\omega_0^2-4\pi^2\epsilon \frac{m^2+\omega_0^2}{|m^2-\omega_0^2|}\sqrt{\alpha_{2m}^2+\beta_{2m}^2}, 
\end{equation*}
so we conclude that the disk cannot be a minimizer.
\end{proof}

\begin{theorem}
Let $B\subset \mathbb{R}^2$ be the unit disk and let $k\in \mathbb{N}$ be an index such that $\mu_k(B)$ is a simple eigenvalue then $B$ is never a local minimizer for the problem:
\begin{equation*}
\inf \{P(\Omega)^2\mu_k(\Omega),\, \Omega \subset \mathbb{R}^2, \text{ bounded, open and convex } \}   
\end{equation*}
\end{theorem}
\begin{proof}
The proof is almost the same as the proof of Theorem \ref{t:ballsingleD}. Using the same perturbation we obtain:
\begin{equation*}
P(\Omega_{\epsilon})^2\mu_k(\Omega_\epsilon)=
4\pi^2\omega_0^2+\epsilon^28\pi^2(\omega_0\omega_2-\alpha_0^2\omega_0^2)+o(\epsilon^2).  
\end{equation*}
We conclude as in the proof of Theorem \ref{t:ballsingleD}.
\end{proof}
\begin{theorem}
The disk $B\subset \mathbb{R}^2$ is never a solution of the following maximization problem 
\begin{equation*}
\sup \{P(\Omega)^2\mu_k(\Omega),\, \Omega \subset \mathbb{R}^2, \text{ bounded, open and convex } \}    
\end{equation*}
\end{theorem}
\begin{proof}
Let $B\subset \mathbb{R}^2$ be the unit disk. The proof is straightforward knowing that the Polya conjecture is true for the disk \cite{FLP23}, indeed we have that for all $k\geq1$ 
\begin{equation*}
\mu_k(B)\leq \frac{4\pi k}{|B|}
\end{equation*}
and this implies
\begin{equation*}
P(B)^2\mu_k(B)\leq 16\pi^2k.
\end{equation*}
Now, consider the rectangle $\Omega_k=[0,k]\times[0,1]$, since its $k$-th eigenvalue is $\mu_k(\Omega_k)=\pi^2$,
we get $P(\Omega_k)^2\mu_k(\Omega_k)=4\pi^2(k+1)^2>16\pi^2k$ for all $k\geq 2$, for $k=1$ we have  $P(\Omega_1)^2\mu_1(\Omega_1)=16\pi^2>4j_{0,1}'^2\pi^2=P(B)^2\mu_k(B)$. 
\end{proof}

\section{Some numerical results}\label{s:numerics}

Multiple shape optimization problems were investigated in the previous sections. Since the optimal shapes are not known, in general, we study numerically the two dimensional case. More precisely, we investigate convex domains of fixed diameter minimizing the Neumann eigenvalues and convex domains of fixed perimeter optimizing the Neumann eigenvalues. 

Numerical shape optimization among convex sets is challenging, since classical domain perturbations methods based on the shape derivative do not preserve convexity. Width or diameter constraints are non-local in nature, rendering the problem more complex. In \cite{Dirichlet-Diameter}, \cite{AntunesBogosel22} spectral decomposition of the support function are used to transform shape optimization problems among convex set into constrained optimization problem using a finite number of parameters. 
Since convex domains having segments in their boundaries correspond to singular support functions, a framework based on discrete approximations of the support function was proposed in \cite{Steklov-Diameter}. This framework was slightly modified and rendered completely rigorous in \cite{Bogosel-Convex}. The simulations presented below are based on \cite{Bogosel-Convex}. In this section, we denote by $p(\theta)$ the support function.

Consider $\theta_j = 2\pi j/N$, $j=0,...,N-1$ a uniform discretization of $[0,2\pi]$. Denoting $h=2\pi/N$, the uniform discretization step, assume approximations $p_i$ of the values of the support function $p(\theta_i)$ verify the constraints:
\begin{equation}
\label{eq:curvature-radii}
\rho_i := \frac{p_{i-1}-2p_i \cos h+p_{i+1}}{2-2\cos h} \geq 0.
\end{equation}
Denoting the radial and tangential directions at $\theta_i$ by $\bo r_i = (\cos \theta_i,\sin \theta_i)$ and $\bo t_i = (-\sin \theta_i,\cos \theta_i)$, consider the polygon given by
\begin{equation}\label{eq:discrete-poly}
\bo x_i := p_i \bo r_i + \frac{p_{i+1}-p_{i-1}}{2\sin h}\bo t_i, i=0,...,N-1.
\end{equation}
In \cite{Bogosel-Convex} it is shown that if constraints \eqref{eq:curvature-radii} are verified then the polygon given by \eqref{eq:discrete-poly} is convex. Moreover, any convex shape can be approximated arbitrarily well in this discrete framework when the number of parameters $N$ discretizing the support function verifies $N \to \infty$.

Width constraints can easily be introduced assuming $N$ is even and imposing
\begin{equation}\label{eq:width-ineq}
0\leq w_i  \leq p_i+p_{i+N/2} \leq W_i, i=0,...,(N-1)/2.
\end{equation}
The numbers $w_i$, $W_i$ represent lower and upper bounds for the width of the shape in the direction $\theta_i$. An upper bound $D$ on the diameter can be imposed by setting $W_i=D$ and $w_i=0$. Prescribing a diameter equal to $D$ is achieved setting $W_i=D$, $i=0,...,(N-1)/2$, $w_0=D$, $w_i=0$, $1\leq i \leq (N-1)/2$.

Consider a shape functional $J(K)$ with shape derivative written in the form $J'(K)(\theta) = \int_{\partial K} \bo f\  \theta \cdot \bo n$. Then, according to \cite{Bogosel-Convex}, we denote by $\phi_i:[0,2\pi]\to \mathbb{R}$ the hat functions which are $2\pi$-periodic, piece-wise affine on $[\theta_i,\theta_{i+1}]$ such that $\phi_i(\theta_j)=\delta_{ij}$. Supposing that $J(K)$ is defined through the parameters $\bo p=(p_0,...,p_N)$ we have, denoting, for simplicity, $K(\bo p)$ the resulting convex shape (given by \eqref{eq:discrete-poly}) and $J(\bo p)=J(K)$
\begin{equation}\label{eq:sensitivity}
     \frac{\partial J(\bo p)}{\partial p_i} = \int_{\partial K(\bo p)} \bo f(\bo x) \phi_i(\theta(\bo x)) d\sigma.
\end{equation}
The angle $\theta(\bo x)$ gives the orientation of the outer normal at the point $\bo x \in \partial K(\bo p)$. The numerical simulations are performed in FreeFEM \cite{FreeFEM}. For the Neumann eigenvalues it is well known that the shape derivative is given by 
\[ J'(K)(\theta) = -\int_{\partial K}(\mu_k(K) u_k^2-|\nabla u_k|^2)\theta\cdot n,  \]
thus $\bo f = (\mu_k(K) u_k^2-|\nabla u_k|^2)$.

\bo{Minimizing the Neumann eigenvalues under diameter constraint.} According to Theorem \ref{thExInfD} there exist optimal shapes solving
\[\min_{D(\Omega)=1} \mu_k(\Omega)\]
when $\Omega$ is convex and $k \geq 2$. Diameter upper bounds $W_i=1$, $i=0,...,(N-1)/2$ are imposed following \eqref{eq:width-ineq}, setting the lower bound $w_0=1$ for one of the directions. Coupled with \eqref{eq:curvature-radii} this gives a set of $N+(N-1)/2+1$ linear inequality constraints for the discrete parameters. Given a set of parameters $p_i$ verifying \eqref{eq:curvature-radii}, a mesh is constructed in FreeFEM for the polygon \eqref{eq:discrete-poly}. The Neumann eigenvalue problem for the Laplacian is solved using $P_2$ finite elements. The sensitivities of the functional with respect to the parameters $p_i$ are computed according to \eqref{eq:sensitivity}. The optimization software IPOPT \cite{IPOPT} is used to solve the constrained optimization problems in FreeFEM. The results obtained are illustrated in Figure \ref{fig:min-diam}

\begin{figure}
    \centering
    \begin{tabular}{cccc}
           \includegraphics[height=0.2\textwidth]{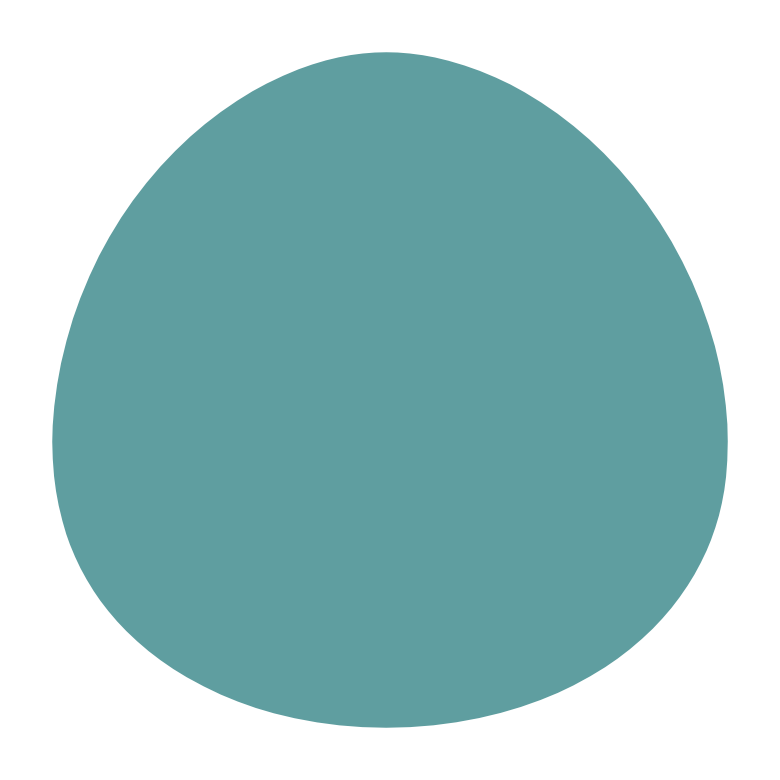}
  &        \includegraphics[height=0.2\textwidth]{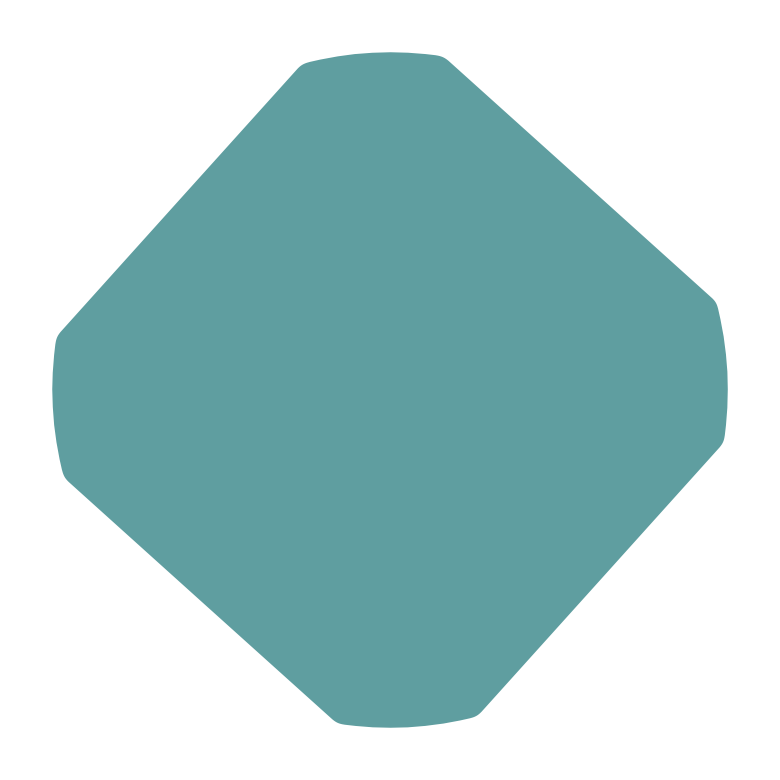}  &        \includegraphics[height=0.2\textwidth]{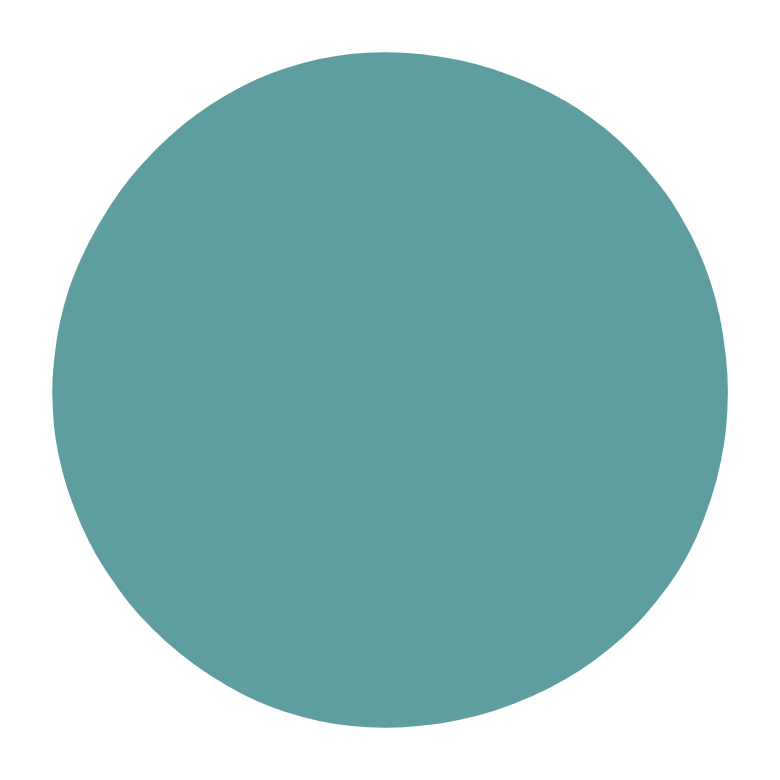} &         \includegraphics[height=0.2\textwidth]{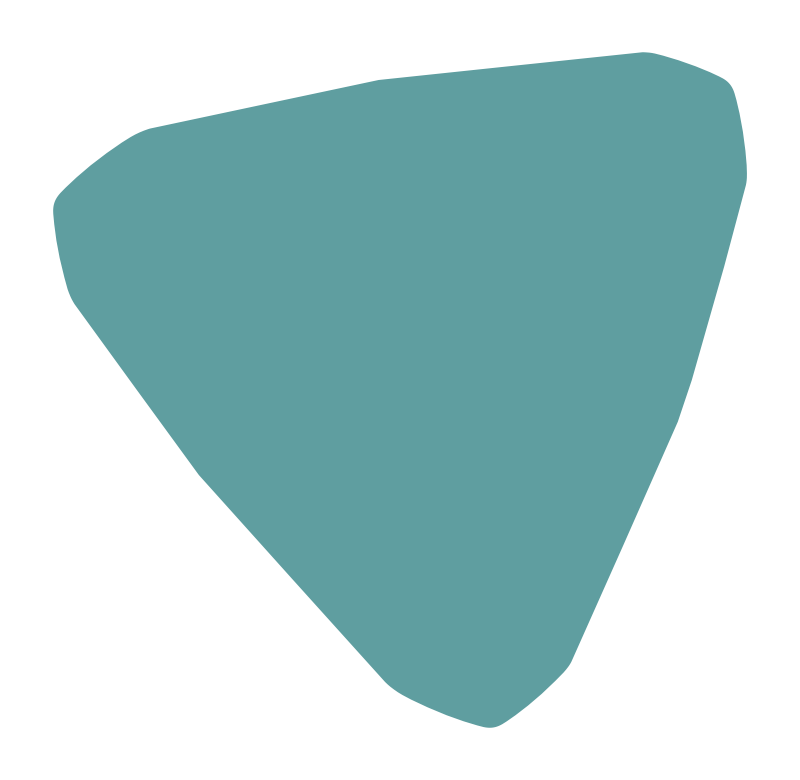} \\
 $k=2$ & $k=3$ & $k=4$ & $k=5$  \\
 $\mu_2 = 13.56$ & $\mu_3=15.42$ & $\mu_4=37.35$ & $\mu_5=48.92$ \\ 
         \includegraphics[height=0.2\textwidth]{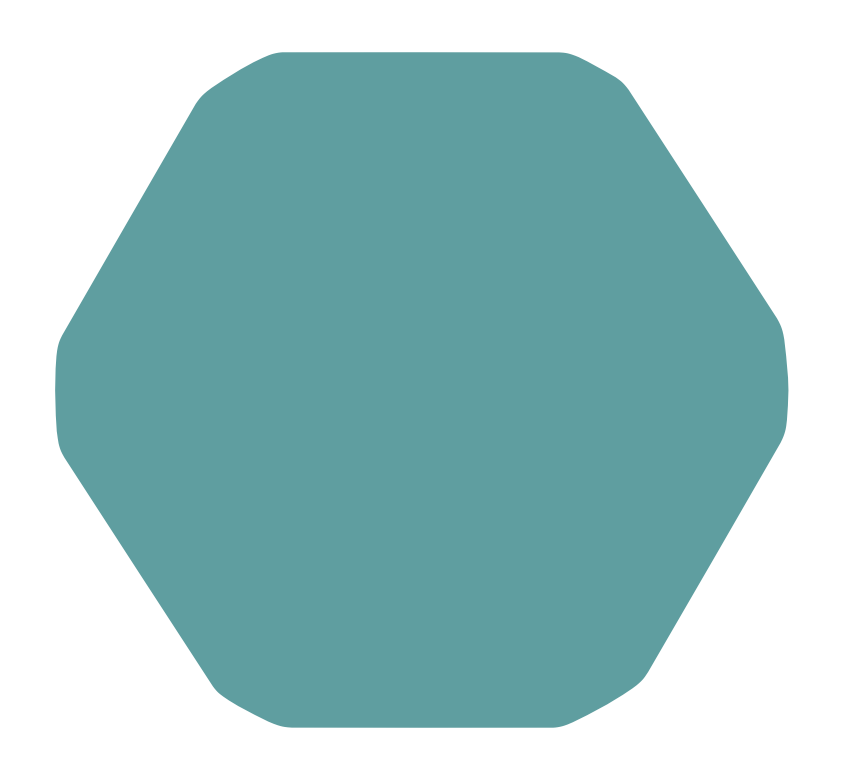}
         &        \includegraphics[height=0.2\textwidth]{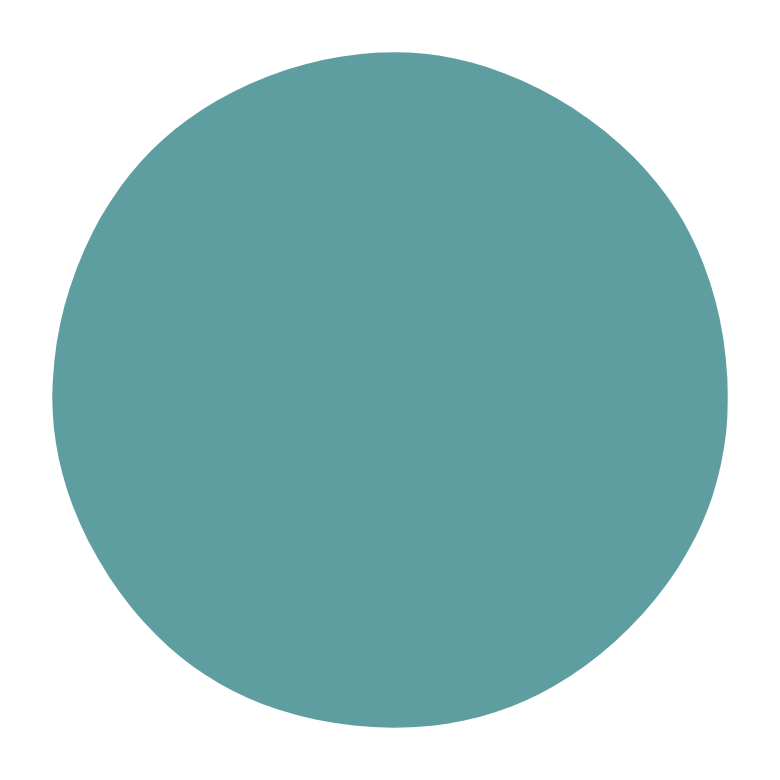}  &        \includegraphics[height=0.2\textwidth]{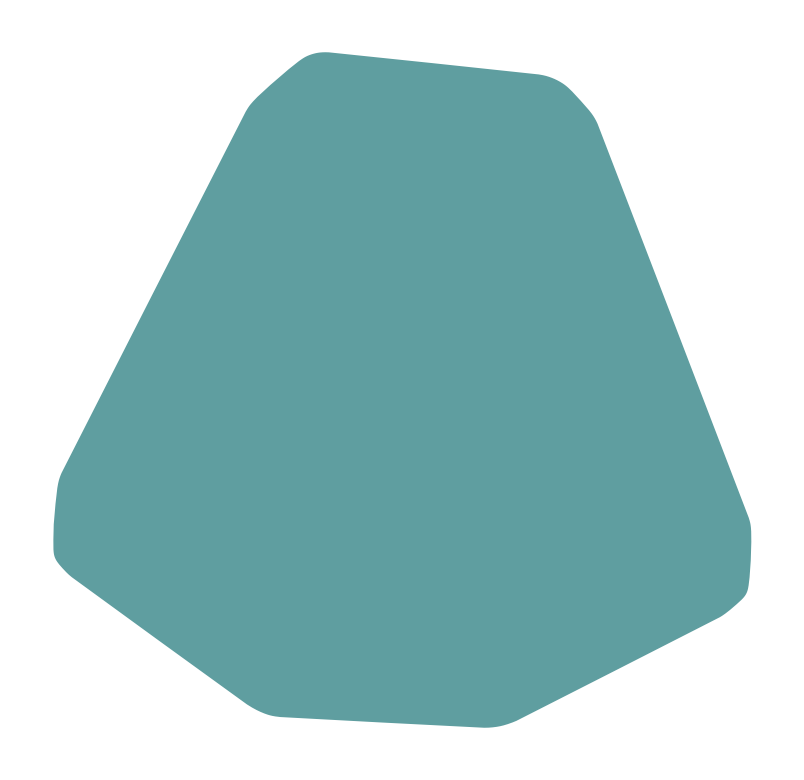} &         \includegraphics[height=0.2\textwidth]{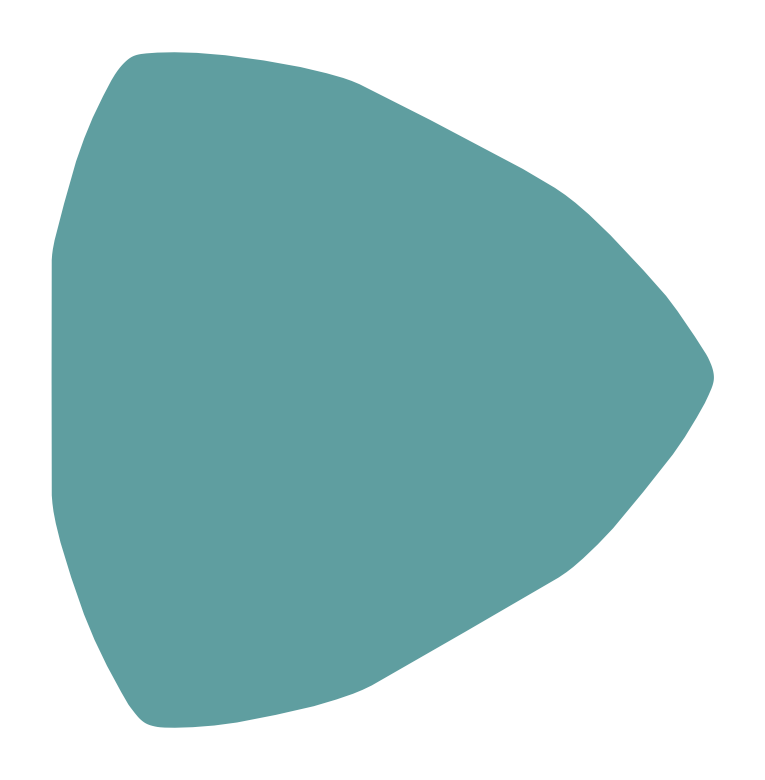}\\  
  $k=6$ & $k=7$ & $k=8$ & $k=9$  \\
   $\mu_6 = 63.49$ & $\mu_7=70.64$ & $\mu_8=97.42$ & $\mu_9=101.70$ \\ 
    \end{tabular}
  \caption{Convex shapes minimizing the $k$-th Neumann eigenvalue for shapes with unit diameter, $2 \leq k \le 9$.}
    \label{fig:min-diam}
\end{figure}

The numerical simulations give the following observations:
\begin{itemize}[noitemsep,topsep=0pt]
	\item The optimal shapes for $k \in \{4,7\}$ seem to be disks.
	\item The optimal shape for $k=2$ seems to have constant width.
	\item In general, points $x$ on the boundary of optimal shapes verify the following: the convexity constraint is saturated at $x$ or the diameter constraint is saturated at $x$. We do not have a theoretical proof of this observation.
\end{itemize}

The multiplicity of eigenvalues at the optimum is often a challenging question. We detail below the observed numerical multiplicity of the eigenvalues of the optimal shapes we obtained.

\begin{itemize}[noitemsep]
    \item[$k=2$:] $\mu_1<\mu_2=\mu_3$
    \item[$k=3$:] $\mu_2<\mu_3<\mu_4$
    \item[$k=4$:] $\mu_3=\mu_4<\mu_5$
    \item[$k=5$:] $\mu_3=\mu_4=\mu_5<\mu_6$
    \item[$k=6$:] $\mu_5<\mu_6<\mu_7$
    \item[$k=7$:] $\mu_6=\mu_7<\mu_8$
    \item[$k=8$:] $\mu_7<\mu_8<\mu_9$
    \item[$k=9$:] $\mu_8=\mu_9<\mu_{10}$
\end{itemize}
Some of the optimal shapes have multiple eigenvalues, but there are also counterexamples: $k \in \{3,6,8\}$. Therefore, no general behavior can be conjectured.

\bo{Minimization/maximization of the Neumann eigenvalues under perimeter constraint.} Theorems \ref{thExInfP}, \ref{thExSupP} imply the existence of optimal shapes minimizing and maximizing the Neumann eigenvalues among convex sets. 

Since in this case we do not have width constraints, we use the dual discretization framework presented in \cite{Bogosel-Convex} using the gauge function. The gauge function associated to a convex body $K$ containing the origin is $\gamma : [0,2\pi] \to \mathbb{R}_+$, $\gamma(\theta) = 1/\rho(\theta)$, where $\rho(\theta) = \sup\{ t : t(\cos \theta,\sin\theta) \in K\}$ is the radial function. Choose $(\theta_i)_{i=0}^{N-1}$ a uniform discretization of $[0,2\pi]$ and denote $\gamma_i = \gamma(\theta_i)$ the values of the gauge functions for direction $\theta_i$. Given $(\gamma_i)_{i=0}^{N-1}$, a sequence of strictly positive parameters, it is straightforward to construct the polygon with vertices $\bo x_i = \frac{1}{\gamma_i} \bo r_i$. This polygon is convex if and only if 
\begin{equation}\label{eq:conv-gauge} \gamma_{i-1}+\gamma_{i+1} - 2\cos h \gamma_i \geq 0, \text{ for every } i=0,...,N-1,\end{equation}
where $h = \frac{2\pi}{N}$. See \cite{Bogosel-Convex} for more details regarding this discretization method.

Given a sequence of parameters $\bo \gamma = (\gamma_i)_{i=0}^{N-1}$, the polygon with vertices $\bo x_i = \frac{1}{\gamma_i} \bo r_i$ is constructed and meshed in FreeFEM. The Neumann eigenvalues for a prescribed index $k \geq 1$ are computed using $P_2$ finite elements. The sensitivities, according to \cite{Bogosel-Convex}, are computed using
\begin{equation}\label{eq:sensitivity-gauge}
\frac{\partial J(\gamma)}{\partial \gamma_i} = -\frac{1}{\gamma_i^2}\int_{\partial K} \bo f(\bo x) \phi_i(\theta(\bo x)) d\sigma,
\end{equation}
where $\phi_i$ are piece-wise affine functions on $[\theta_i,\theta_{i+1}]$ verifying $\phi_i(\theta_j) = \delta_{ij}$ and $\bo f = (\mu_k(K) u_k^2-|\nabla u_k|^2)$.

The scale invariant functional
\[ P(\Omega)^2\mu_k(\Omega)\]
is used, with positivity constraints $\gamma_i>0$ and convexity constraints given by \eqref{eq:conv-gauge}. Minimizers are shown in Figures \ref{fig:perim-minimizers} and maximizers in Figure \ref{fig:perim-maximizers}. We have the following observations:
\begin{itemize}
	\item The maximizer of $\mu_1$ under perimeter constraint found numerically is the square. The best numerical value found was $157.79$, while the exact value for the square is $16\pi^2 \approx 157.91$. In \cite{HLL} is shown that the equilateral triangle has the same objective value. The equilateral triangle was not recovered numerically even though it has the same value for the objective function. Moreover, when imposing symmetry constraints the square and the equilateral triangle are the only maximizers. 
	\item The maximizer of $\mu_2$ under perimeter constraint seems to be a rectangle with one side equal to twice the other one. The best numerical value attained is $353.48$, while the analytical value for a $2 \times 1$ rectangle (rescaled to have unit perimeter) is $36\pi^2 \approx 355.30.$ Moreover, all maximizers under perimeter constraint seem to be polygons.
	\item The minimizers, on the other hand, are convex sets which do not have corners and which may contain segments in their boundaries.
 \item The shapes minimizing or maximizing $\mu_k$ under perimeter constraints seem to have multiple eigenvalues. This is a classical behavior in spectral optimization. More precisely, for minimizers the multiplicity cluster ends at $\mu_k$ (i.e. $\mu_k(\Omega)$ is multiple and $\mu_k(\Omega) <\mu_{k+1}(\Omega)$), while for maximizers the opposite holds: the multiplicity cluster starts with $\mu_k(\Omega)$ (i.e. $\mu_k(\Omega)$ is multiple and $\mu_{k-1}(\Omega)<\mu_k(\Omega)$). For all computations shown in Figures \ref{fig:perim-minimizers}, \ref{fig:perim-maximizers} the optimal eigenvalues are double. 
\end{itemize}

\begin{figure}
	\centering
	\begin{tabular}{cccc}
		\includegraphics[height=0.2\textwidth]{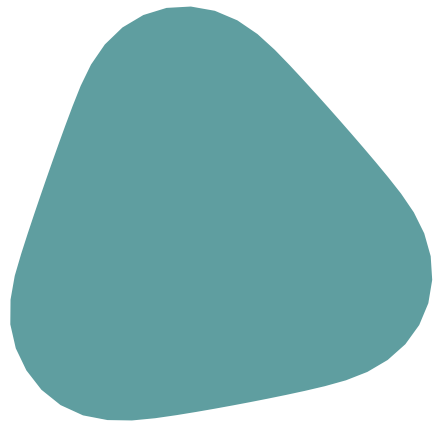}  &
		\includegraphics[height=0.2\textwidth]{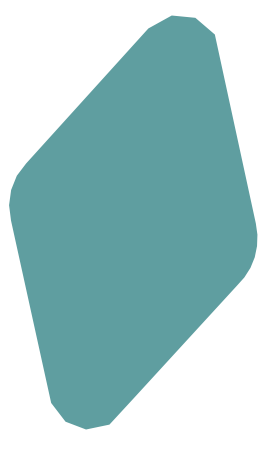}  &
		\includegraphics[height=0.2\textwidth]{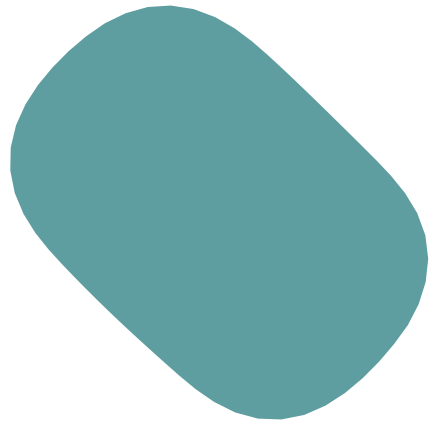}  &
            \includegraphics[height=0.2\textwidth]{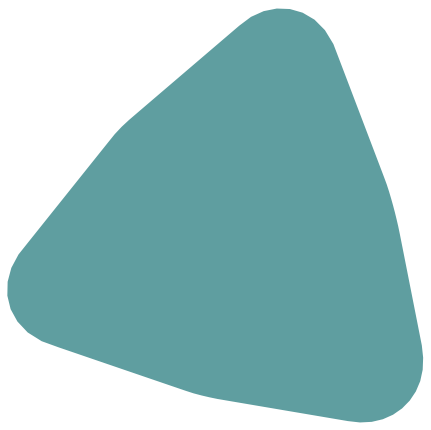} \\
		$\mu_2=132.07$ & $\mu_3=256.52$ & $\mu_4=358.57$ & 
        $\mu_5=391.53$\\
            \includegraphics[height=0.2\textwidth]{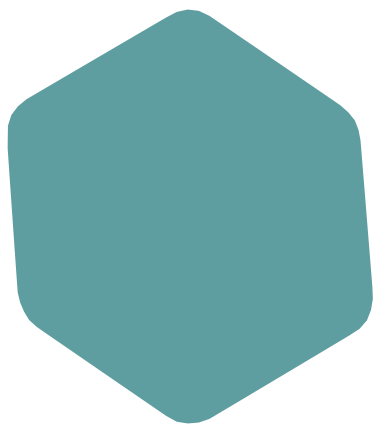}  &
		\includegraphics[height=0.2\textwidth]{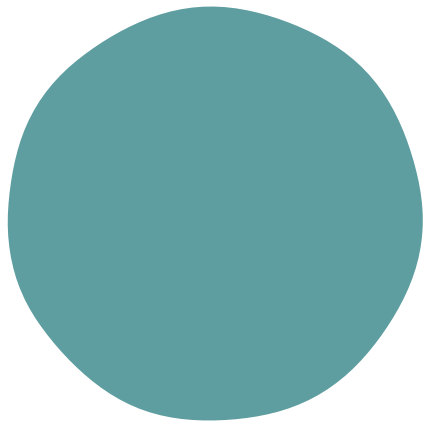}  &
		\includegraphics[height=0.2\textwidth]{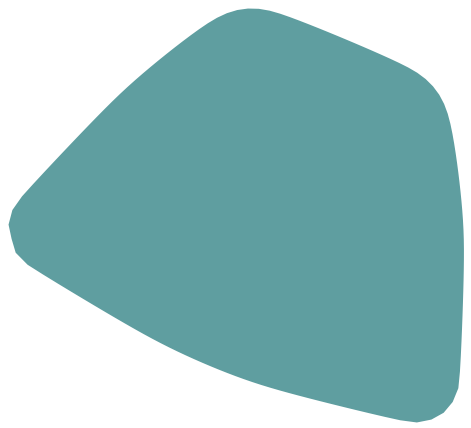}  &
            \includegraphics[height=0.2\textwidth]{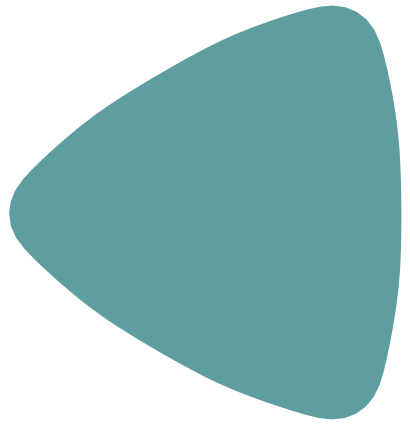} \\
		$\mu_6=616.83$ & $\mu_7=697.44$ & $\mu_8=863.53$ & 
        $\mu_9=985.59$\\
	\end{tabular}
\caption{Minimizers for the $k$-th Neumann eigenvalue among shapes with unit perimeter.}
\label{fig:perim-minimizers}
\end{figure}

\begin{figure}
	\centering
	\begin{tabular}{ccc}
		\includegraphics[height=0.25\textwidth]{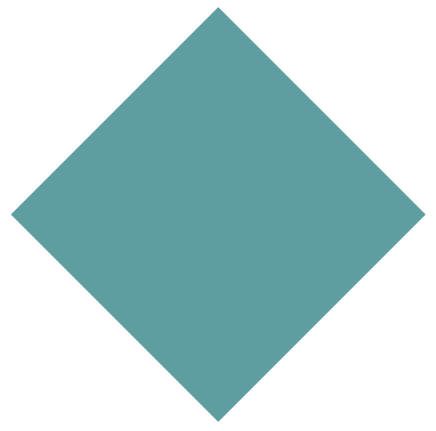}  &
		\includegraphics[height=0.25\textwidth]{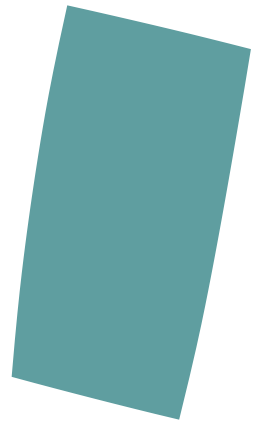}  &
		\includegraphics[height=0.25\textwidth]{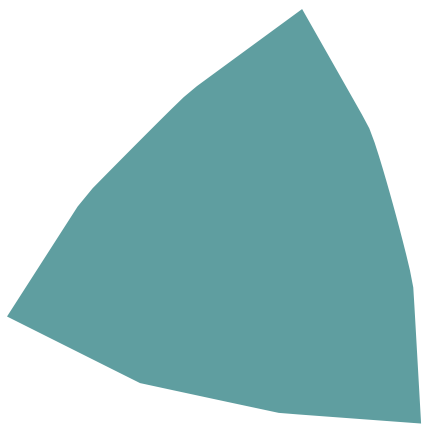}  
  \\
	    $\mu_1=157.91$ &	$\mu_2=353.48$ & $\mu_3=492.45$ 
	\end{tabular}
	\caption{Maximizers for the $k$-th Neumann eigenvalue among shapes with unit perimeter.}
	\label{fig:perim-maximizers}
\end{figure}

We refrain from making any conjectures regarding multiplicities of Neumann eigenvalues of optimal shapes under convexity and perimeter constraints due to the following considerations:
\begin{itemize}
    \item When minimizing Dirichlet-Laplace eigenvalues $(\lambda_k(\Omega))_{k \geq 1}$ under area constraint, optimal shapes found numerically always have multiple eigenvalues \cite{antunes-freitas}. The corresponding theoretical question is still open. Finding a counterexample would contradict a conjecture due to Schiffer or Bernstein: if $-\Delta u = \lambda u$ in $\Omega$, $u=0$ on $\partial \Omega$ and $\partial_n u=c$ on $\partial \Omega$ then $\Omega$ is a disk. For further details see the discussion in \cite[Section 4.3]{bogosel-oudet-perim}
    \item Already when minimizing $\lambda_2$ with convexity and area constraints in dimension two, since first eigenvalues of connected domains are simple, we have $\lambda_1(\Omega^*)<\lambda_2(\Omega^*)$. The multiplicity is lost when the convexity constraint is added. Note that the convexity constraint is saturated when minimizing $\lambda_2$ since segments are present in the boundary. See \cite{henrot-oudet-second-eig} for more details in this case.
    \item When minimizing the Dirichlet-Laplace eigenvalues with perimeter constraint \cite{bogosel-oudet-perim} there exist instances where the optimal shape has a simple eigenvalue at the optimum. 
\end{itemize}

\bibliographystyle{abbrv}
\bibliography{Refer1}
\end{document}